\documentclass[12pt]{article}

\usepackage[a4paper,margin=27mm]{geometry}
\usepackage[T1]{fontenc}
\usepackage{lmodern}
\usepackage{microtype}
\usepackage{amsmath,amssymb,amsfonts,amsthm,mathtools,mathrsfs}
\usepackage{bm}
\usepackage{booktabs}
\usepackage{graphicx}
\usepackage{enumitem}
\usepackage{xcolor}
\usepackage{hyperref}
\usepackage[nameinlink,noabbrev]{cleveref}
\usepackage{fancyhdr}
\usepackage{setspace}
\usepackage{verbatim}
\usepackage{subcaption}
\usepackage{tikz-cd}
\hypersetup{
	colorlinks=true,
	linkcolor=blue!50!black,
	citecolor=blue!50!black,
	urlcolor=blue!50!black,
	pdftitle={Spectral Convergence of the Multipole Expansion Method for Acoustic Scattering in Three Dimensions},
	pdfauthor={Jinrui Zhang and Jun Lai}
}
\graphicspath{{figure/}}
\allowdisplaybreaks[2]

\newtheorem{theorem}{Theorem}[section]
\newtheorem{lemma}[theorem]{Lemma}
\newtheorem{proposition}[theorem]{Proposition}

\newtheorem{assumption}[theorem]{Assumption}
\theoremstyle{definition}

\newtheorem{remark}[theorem]{Remark}

\newcommand{\Sph}{\mathbb S^2}
\newcommand{\R}{\mathbb R}
\newcommand{\seqsp}{\mathcal{H}}
\newcommand{\inc}{\mathrm{inc}}
\newcommand{\sca}{\mathrm{s}}
\newcommand{\diff}{\mathrm{diff}}
\newcommand{\Id}{\mathbb I}
\newcommand{\bA}{\mathbb A}
\newcommand{\bD}{\mathbb D}
\newcommand{\bW}{\mathbb W}
\newcommand{\PN}{\mathbb P_N}
\newcommand{\QN}{\mathbb Q_N}
\newcommand{\HS}{\mathrm{HS}}
\newcommand{\dd}{\,\mathrm d}
\newcommand{\e}{\mathrm e}
\newcommand{\ii}{\mathrm i}
\newcommand{\cO}{\mathcal O}
\newcommand{\norm}[1]{\left\lVert#1\right\rVert}
\newcommand{\abs}[1]{\left\lvert#1\right\rvert}

\title{\bfseries Spectral Convergence of the Multipole Expansion Method for Acoustic Scattering in Three Dimensions}
\author{
	Jinrui Zhang\thanks{Department of Mathematics, The Hong Kong University of Science and Technology,	Clear Water Bay, Kowloon, Hong Kong Special Administrative Region of China. Email: \texttt{jinruizhang@ust.hk}.}
	\and
	Jun Lai\thanks{School of Mathematical Sciences and Center for Interdisciplinary Applied Mathematics, Zhejiang University, Hangzhou, Zhejiang 310027, China. Email: \texttt{laijun6@zju.edu.cn}.}
}
\date{}

\begin{document}
	\maketitle

	 \begin{abstract}
	 	Multiple scattering is a fundamental wave interaction phenomenon in acoustics and electromagnetics. The multipole expansion method (MEM) is the basis of many fast algorithms, such as the fast multipole method (FMM), for such problems. However, due to the infinitely many wave reflections involved, its convergence in three dimensions remains unexplored. In this paper, we prove spectral convergence of the MEM for time-harmonic acoustic scattering by finitely many well-separated spheres in three dimensions. Using a diagonally preconditioned single-layer formulation, we analyze the degree-$N$ truncated system in a natural spherical harmonic energy space. We split the interaction truncation into target-side and source-side high-degree parts and choose a different representation for each: a projected Green-kernel representation for the former and degree-wise estimates of a translated spherical wave family for the latter. The resulting argument, based on Parseval's identity and the spherical harmonic addition theorem, provides a general framework for the convergence analysis of MEM and reveals the geometric and physical origins of the convergence factors. We also obtain a sharper estimate through the first-transfer analysis. Numerical experiments confirm the predicted spectral decay and the geometric convergence factor. This paves the way for the convergence analysis of a large class of fast algorithms for multiple scattering.
	 \end{abstract}
     
	\noindent\textbf{Keywords.}
	Multiple scattering, multipole expansion method, spherical harmonics, convergence analysis.
	
	\noindent\textbf{Mathematics Subject Classification.}
	65N12, 65N15, 78M16, 35J05
	
	\section{Introduction}\label{sec:introduction}
	
	Multiple scattering is fundamental in acoustics, electromagnetics,
	elastodynamics, photonics, and the analysis of composite and metamaterial
	media~\cite{ColtonKress2019, Martin2006}. Depending on
	the geometry and frequency regime, it is commonly treated by finite or
	boundary element discretizations \cite{barucqNumericalRobustnessSinglelayer2018,ganesh2009high}, domain decomposition schemes \cite{grote2004dirichlet,xieEfficientIterativeMethod2020}, fast translation algorithms \cite{gimbutasComputationalSoftwareSimple2015}, and Foldy-Lax approximations \cite{huangEfficientAlgorithmGeneralized2013}. A detailed literature review can be found in \cite{Martin2006}. For
	configurations containing many circular or spherical particles, as illustrated in Figure~\ref{Geometryillustration}, separated wave
	representations are especially attractive because they exploit the exact local
	geometry and replace surface meshes by interactions among modal coefficients.
	
	The multipole expansion method (MEM) is the classical realization of this
	idea. It originated in early treatments of cylinder arrays and
	multiple-cylinder scattering \cite{Row1955,Zavisca1913} and later became
	closely connected with T-matrix formulations for scatterers of general shape \cite{caiLargescaleMultipleScattering1999,mishchenkoComprehensiveThematicTmatrix2020a,Waterman1965}. An outgoing separated
	expansion is attached to each particle, translated to every other center, and
	coupled through the boundary conditions. Once the T-matrix is known, the method is meshless, and high accuracy can often be achieved
	with comparatively few unknowns, and the fast multipole method (FMM) acceleration can be incorporated for large-scale systems. MEM formulations and their fast variants are
	widely used for acoustic spheres
	\cite{GumerovDuraiswami2002,
	GumerovDuraiswami2005,KocChew1998}, electromagnetic particles
	\cite{gimbutasFastMultiparticleScattering2013,Waterman1965}, and elastic
	scatterers \cite{laiFrameworkSimulationMultiple2019,laiFastInverseElastic2022}. Despite this long history and broad
	use, to the best of our knowledge, a rigorous convergence theory for truncating the fully coupled
	three-dimensional MEM has remained unavailable.
	
	Several related approximation problems have been analyzed. Estimates for
	truncated spherical addition theorems and fast multipole translations are given
	in \cite{AminiProfit2000,Darve2000,KocSongChew1999,Meng2024,ohnukiTruncationErrorAnalysis2004}, and convergence
	of a reduced-basis acoustic T-matrix approximation was established in
	\cite{GaneshHawkinsHiptmair2012}. Computational studies of multisphere
	multipole systems also show rapid convergence
	\cite{GumerovDuraiswami2002,Lee2015}. Spectral and condition-number estimates for low-frequency multiple scattering in dilute and dense media are given in \cite{antoineSpectralConditionNumber2013,thierrySpectralConditionNumber2013}. These results address individual
	translations, fast summation procedures, or approximations of scattering
	responses in specific settings. However, they do not establish either stability or error estimates for the globally
	coupled degree-$N$ MEM system, in which truncation at each expansion center is
	propagated through arbitrarily many reflections.
	
	The closest related work is the two-dimensional convergence analysis in
	\cite{Fitzpatrick2021}, which provides explicit full and first-transfer convergence factors for both plane wave and point source incidence. Its proof starts from explicit
	cylindrical translation coefficients, estimates coupled double-index sums, and
	reduces them to hypergeometric functions whose large-order asymptotics require
	highly involved analysis. The authors explicitly identified the extension to spheres as an
	open direction and observed that associated Legendre functions and more
	complicated addition theorems could produce hypergeometric expressions too
	difficult to control. The issue is structural: in three dimensions, a spherical harmonic mode of degree $n$
	has $2n+1$ orders, and translation couples $(n,m)$ to $(\ell,r)$
	through Gaunt coefficients
	\cite{EptonDembart1995,GumerovDuraiswami2003,DLMF}, making term-by-term estimates prohibitively tedious, if not intractable.
	
	Our main finding is that these translation coefficients need not be estimated
	individually. In particular, after diagonal preconditioning, we decompose the discarded
	interaction into target-side and source-side high-degree
	parts. For the target-side part, projecting the Helmholtz Green kernel onto a
	fixed output mode cancels the regular Bessel factor in the self block and leaves
	a smooth kernel controlled by a Hankel ratio. For the source-side part, fixing
	an input mode produces a translated outgoing degree-$n$ spherical-wave family.
	The inverse self block is handled by estimating its value and gradient
	energies. In both directions, Parseval's identity and the spherical harmonic
	addition theorem sum all orders before the large-degree estimate is
	taken. Only polynomial factors remain, so the three-dimensional multiplicity
	does not alter the geometric $N$th-root rate. This avoids direct estimates of the Gaunt coefficients and
	hypergeometric calculations anticipated by the two-dimensional analysis~\cite{Fitzpatrick2021}.
	
	In fact, our approach based on the target/source decomposition is not restricted to three dimensions.
	For scattering by circular obstacles in two dimensions, it gives the two directed interaction tails directly
	from a Green kernel row representation and the physical translated wave in a
	column representation, bypassing the coupled double-index sums of
	\cite{Fitzpatrick2021}. The argument therefore supplies a concise framework in both
	two and three dimensions. It also explains the rates geometrically: the full factor
	$a_p/(d_{pq}-a_q)$ compares the target radius $a_p$ with the distance from its center
	to the nearest point of the source particle, whereas the sharper first-transfer
	factor is set by the larger analyticity domain of the isolated-particle field.
	Thus the convergence factors arise from propagation geometry and analyticity,
	not only from cancellation in special-function formulas.
	
	For the analysis we use the coefficients of a single-layer density
	rather than the conventional outgoing multipole amplitudes. This choice
	makes the natural Sobolev structure and the interaction tails particularly
	transparent. At every finite truncation degree, the density coefficients
	are related diagonally to the standard outgoing multipole coefficients, so
	the resulting truncated system is the same MEM discretization written in a
	different set of variables. The
	single-layer nonresonance assumption used below is therefore a restriction
	of this analytical representation, not of the underlying exterior
	scattering problem or of the classical Mie system. As pointed
	out in the two-dimensional
	analysis \cite{Fitzpatrick2021}, changing the boundary condition or integral
	representation is expected to alter only subexponential factors, provided
	the corresponding modal self block bounds are available. Moreover, the same argument extends to any fixed spherical Sobolev scale
	$\seqsp^s$, with the additional Sobolev weights affecting only
	polynomial factors.
	The framework also suggests a route toward combining local T-matrix approximations \cite{GaneshHawkinsHiptmair2012} with a global truncation analysis for particles of general shape. In this setting, the coupled convergence rate would depend jointly on the inter-particle translation tails and the local T-matrix truncation errors, with the latter influenced by particle size, shape, and boundary regularity.
	
	The paper is organized as follows. Section~\ref{sec:formulation} formulates
	the scattering problem and the infinite and truncated MEM systems.
	Section~\ref{sec:preparations} gives the exact truncation error identity and the
	degree-wise estimates used later. Section~\ref{sec:full-analysis} proves the
	interaction and incident data tail bounds, stability of the truncated systems,
	and the full convergence theorem. Section~\ref{sec:first} derives the sharper
	first-transfer estimate and compares the framework with the two-dimensional
	theory. Section~\ref{sec:numerics} presents numerical verification, and
	Section~\ref{sec:conclusion} concludes the paper. The spherical harmonic identities are
	collected in Appendix~\ref{app:spherical-identities}.
	
	\section{Formulation of multiple scattering by MEM}\label{sec:formulation}
	

	\begin{figure}[htbp]
  \centering
  \begin{tikzpicture}[
      x={(1.15cm,0cm)},
      y={(-0.42cm,0.32cm)},
      z={(0cm,0.88cm)},
      wave/.style={black,thick,->,>=stealth}
    ]
    \draw[wave]
      (-3.0cm,0.75cm) .. controls (-2.65cm,1.10cm) and (-2.30cm,0.40cm) ..
      (-1.95cm,0.75cm) .. controls (-1.70cm,1.00cm) and (-1.48cm,0.63cm) ..
      (-1.28cm,0.83cm);
    \draw[wave]
      (-3.0cm,1.70cm) .. controls (-2.65cm,2.05cm) and (-2.30cm,1.35cm) ..
      (-1.95cm,1.70cm) .. controls (-1.70cm,1.95cm) and (-1.48cm,1.58cm) ..
      (-1.28cm,1.78cm);
    \node[black,left] at (-2.55cm,1.23cm) {$u^{\mathrm{inc}}$};

    \draw[wave]
      (2.55cm,0.45cm) .. controls (2.85cm,0.68cm) and (3.10cm,0.12cm) ..
      (3.40cm,0.28cm) .. controls (3.70cm,0.42cm) and (3.98cm,-0.12cm) ..
      (4.28cm,-0.28cm);
    \draw[wave]
      (2.60cm,1.15cm) .. controls (2.90cm,0.88cm) and (3.15cm,1.45cm) ..
      (3.45cm,1.18cm) .. controls (3.75cm,0.91cm) and (4.00cm,1.48cm) ..
      (4.33cm,1.22cm);
   
    \draw[wave]
      (2.55cm,1.85cm) .. controls (2.85cm,1.65cm) and (3.10cm,2.18cm) ..
      (3.40cm,2.02cm) .. controls (3.70cm,1.88cm) and (3.98cm,2.42cm) ..
      (4.28cm,2.58cm);
    \node[black,right] at (2.95cm,2.40cm) {$u^{\mathrm{s}}$};

    \foreach \y in {2,1,0}{
      \foreach \z in {2,1,0}{
        \foreach \x in {0,1,2}{
          \pgfmathsetmacro{\ballradius}{0.145 + 0.015*mod(2*\x+3*\y+\z,4)}
          \pgfmathsetmacro{\jx}{0.10*sin(37*\x+71*\y+113*\z)}
          \pgfmathsetmacro{\jy}{0.10*sin(83*\x+29*\y+59*\z+20)}
          \pgfmathsetmacro{\jz}{0.10*sin(47*\x+97*\y+31*\z+50)}
          \shade[ball color=blue!65]
            ({\x+\jx},{\y+\jy},{\z+\jz}) circle[radius=\ballradius cm];
        }
      }
    }
  \end{tikzpicture}
  \caption{Wave scattering of multiple well-separated spheres.}
  \label{Geometryillustration}
\end{figure}
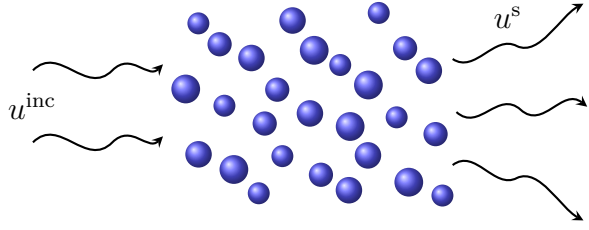

	Consider the multiple acoustic scattering of $M$ distinct spheres, as illustrated in Figure \ref{Geometryillustration}. More precisely, let
	\[
	B_p:=\{\bm{x}\in\R^3:\abs{\bm{x}-\bm{c}_p}<a_p\},
	\qquad
	\Gamma_p:=\partial B_p,
	\qquad p=1,\ldots,M,
	\]
	where $\bm{c}_p\in\R^3$ and $a_p>0$ denote the center and radius of the sphere
	$\Gamma_p$, respectively. Put $d_{pq}:=\abs{\bm{c}_p-\bm{c}_q}$ and assume strict disjointness,
	\begin{equation}\label{eq:disjoint}
		d_{pq}>a_p+a_q,
		\qquad p\ne q.
	\end{equation}
	The exterior domain is
	$\Omega^+:=\R^3\setminus\bigcup_{p=1}^M\overline{B_p}$.
	For a given wavenumber $k>0$, the sound-soft scattered field satisfies
	\begin{equation}\label{eq:pde}
		\begin{aligned}
			\begin{cases}
				(\Delta+k^2)u^{\sca}=0 &\text{in }\Omega^+,\\
				u^{\sca}=-u^{\inc} &\text{on }\Gamma:=\bigcup_{p=1}^M\Gamma_p,\\
				\lim_{r\to\infty}r(\partial_ru^{\sca}-\ii ku^{\sca})=0,
				&r=\abs{\bm{x}}.
			\end{cases}
		\end{aligned}
	\end{equation}
    Note that here the sound-soft boundary condition is not essential. Other types of boundary condition, such as sound-hard or impedance conditions, can be analyzed analogously. For the incident wave, we consider either a plane wave
	\begin{equation}\label{eq:plane-wave}
		u^{\inc}(\bm{x})=\e^{\ii k\widehat{\bm d}\cdot\bm{x}},
		\qquad \widehat{\bm d}\in\Sph,
	\end{equation}
	or a point source
	\begin{equation}\label{eq:point-source}
		u^{\inc}(\bm{x})=\Phi_k(\bm{x},\bm{x}_0)
		:=\frac{\e^{\ii k\abs{\bm{x}-\bm{x}_0}}}{4\pi\abs{\bm{x}-\bm{x}_0}},
		\qquad \bm{x}_0\in\Omega^+.
	\end{equation}

	Such a multiple scattering problem can be solved by a variety of numerical
	techniques~\cite{barucqNumericalRobustnessSinglelayer2018,ganesh2009high}. However, for configurations of spherical particles, the MEM~\cite{Martin2006} is among the most classical and effective approaches. It avoids surface meshing, and reduces inter-particle coupling to translations among modal coefficients. 

	
	
	To introduce the main idea of MEM, we denote 
	$\bm{\omega}=(\sin\theta\cos\varphi,\sin\theta\sin\varphi,\cos\theta)\in\Sph$.
	With the associated Legendre functions in the convention of
	\cite{DLMF}, define the complex spherical harmonics by
	\begin{equation}\label{eq:Y-definition}
		Y_n^m(\theta,\varphi)
		:=
		\left(
		\frac{2n+1}{4\pi}\frac{(n-m)!}{(n+m)!}
		\right)^{1/2}
		P_n^m(\cos\theta)\e^{\ii m\varphi},
		\qquad n\geq0,\quad -n\leq m\leq n.
	\end{equation}
	They satisfy
	\begin{equation}\label{eq:Y-orthogonal}
		\int_{\Sph}Y_n^m(\bm{\omega})\overline{Y_{\ell}^{r}(\bm{\omega})}\dd\bm{\omega}
		=\delta_{n\ell}\delta_{mr}.
	\end{equation}
	Here and below, $n$ is the spherical harmonic \emph{degree} and $m$ is the
	\emph{order}. We use ``the degree-$n$ family'' to mean all $2n+1$
	modes $\{Y_n^m\}_{m=-n}^n$ at a fixed degree, and ``degree-wise'' when an
	estimate is performed after summing over the orders.
	On $\Gamma_p$, let us define
	\begin{equation}\label{eq:surface-basis}
		b_{p,nm}(\bm{c}_p+a_p\bm{\omega}):=a_p^{-1}Y_n^m(\bm{\omega}).
	\end{equation}
	Since $\dd s=a_p^2\dd\bm{\omega}$, the family $\{b_{p,nm}\}$ is orthonormal in $L^2(\Gamma_p)$.
	
	For $f\in\mathcal D'(\Gamma_p)$, let
	\[
	f_{nm}:=\langle f,b_{p,nm}\rangle
	\]
	denote its spherical-harmonic coefficients, with the pairing
	interpreted by duality whenever necessary. Put
	\[
	\langle n\rangle
	:=
	\bigl(1+n(n+1)\bigr)^{1/2}.
	\]
	For $s\in\R$, we use the equivalent spectral Sobolev norm
	\begin{equation}\label{eq:single-sphere-energy-norm}
		\|f\|_{H^s(\Gamma_p)}^2
		:=
		\sum_{n=0}^{\infty}
		\sum_{m=-n}^{n}
		\langle n\rangle^{2s}|f_{nm}|^2.
	\end{equation}
	For each fixed radius $a_p$, this norm is equivalent to the standard Sobolev
	norm on $\Gamma_p$. In particular, the natural density space below uses
	$s=-1/2$. 
	We then define the product energy space
	\begin{equation}\label{eq:hs-definition}
		\seqsp^{-1/2}:=\bigoplus_{p=1}^M H^{-1/2}(\Gamma_p),
	\end{equation}
	with
	\[
	\|\Phi\|_{\seqsp^{-1/2}}^2
	:=
	\sum_{p=1}^M
	\sum_{n=0}^{\infty}
	\sum_{m=-n}^{n}
	\langle n\rangle^{-1}
	|\Phi_{p,nm}|^2.
	\]
	For each fixed collection of radii, this norm is equivalent to the product Sobolev norm on $\bigoplus_pH^{-1/2}(\Gamma_p)$.
	The properties of spherical harmonic and Green's function identities used throughout the
	analysis are collected in Appendix~\ref{app:spherical-identities}.
	
	
	Define the single-layer potential and its Dirichlet trace by
	\begin{equation}\label{eq:single-layer}
		(\mathcal S\phi)(\bm{x}):=\int_{\Gamma}\Phi_k(\bm{x},\bm{y})\phi(\bm{y})\dd s_{\bm y},
		\qquad
		V\phi:=\gamma\mathcal S\phi.
	\end{equation}
	We use the representation $u^{\sca}=\mathcal S\phi$ for the solution of the boundary-value problem \eqref{eq:pde}, which leads to the boundary integral equation
	\begin{equation}\label{eq:bie}
		V\phi=-u^{\inc}|_{\Gamma}.
	\end{equation}
	Let $V_{pq}$ denote the block of $V$ mapping densities on $\Gamma_q$ to
	traces on $\Gamma_p$. Using the spherical wave addition formula (see equation \eqref{eq:green-expansion} in the Appendix), we obtain
	\begin{equation}\label{eq:self-eigenvalue}
		V_{pp}b_{p,nm}=\lambda_n^{(p)}b_{p,nm},
		\qquad
		\lambda_n^{(p)}=\ii k a_p^2j_n(ka_p)h_n^{(1)}(ka_p).
	\end{equation}
	For fixed $z>0$, recall the standard large-order asymptotics for Bessel functions 
	\cite{DLMF}
	\begin{align}
		j_n(z)&=\frac{z^n}{(2n+1)!!}\bigl(1+\cO(n^{-1})\bigr),\label{eq:j-asymptotic}\\
		h_n^{(1)}(z)&=-\ii\frac{(2n-1)!!}{z^{n+1}}
		\bigl(1+\cO(n^{-1})\bigr).\label{eq:h-asymptotic}
	\end{align}
	Consequently, it holds
	\begin{equation}\label{eq:lambda-asymptotic}
		\lambda_n^{(p)}
		=\frac{a_p}{2n+1}\bigl(1+\cO(n^{-1})\bigr).
	\end{equation}
	However, this asymptotic relation does not rule out the possibility that
	$\lambda_n^{(p)}=0$. To ensure that $V_{pp}$ is invertible, we impose the
	following nonresonance assumption.
	\begin{assumption}[Single-layer nonresonance]\label{ass:nonresonance}
		For every $p$ and $n$, $j_n(ka_p)\ne0$. Equivalently, $k^2$ is not an interior Dirichlet eigenvalue of any ball $B_p$.
	\end{assumption}
    \begin{remark}
        The nonresonance assumption is needed only for the single-layer representation \eqref{eq:single-layer}. It can be removed by using a combined layer potential~\cite{ColtonKress2019}. However, the analysis becomes slightly lengthier, although the main idea remains the same. 
    \end{remark}  
    
	 Under the nonresonance assumption, diagonal preconditioning by the isolated sphere self-interaction gives the
	operator form used in the convergence analysis. Specifically, let
	\begin{equation}\label{eq:D-and-A}
		\bD:=\operatorname{diag}(V_{11},\ldots,V_{MM}),
		\qquad
		\bW:=\bD^{-1}V=\Id+\bA,
	\end{equation}
	where
	\begin{equation}\label{eq:A-block}
		A_{pp}=0,
		\qquad
		A_{pq}:=V_{pp}^{-1}V_{pq}\quad(p\ne q).
	\end{equation}
	The right-hand side is
	\begin{equation}\label{eq:G-definition}
		G:=-\bD^{-1}u^{\inc}|_{\Gamma}.
	\end{equation}
	Then the infinite density system for MEM is
	\begin{equation}\label{eq:infinite-system}
		(\Id+\bA)\Phi=G.
	\end{equation}

	We next establish the well-posedness of the infinite system
	\eqref{eq:infinite-system}. The analysis begins with the isomorphism property
	of the self-interaction blocks.

	\begin{lemma}\label{lem:self-inverse}
		Under Assumption \ref{ass:nonresonance}, there is $C_p>0$ such that
		\begin{equation}\label{eq:self-inverse-bound}
			\norm{V_{pp}^{-1}f}_{H^{-1/2}(\Gamma_p)}
			\leq C_p\norm{f}_{H^{1/2}(\Gamma_p)}.
		\end{equation}
		Moreover, $V_{pp}$ is an isomorphism from $H^{t}(\Gamma_p)$ to $H^{t+1}(\Gamma_p)$ for any $t\in \R$.
	\end{lemma}
	
	\begin{proof}
 Expand $f=\sum_{n,m}f_{nm}b_{p,nm}$. By equation \eqref{eq:self-eigenvalue}, we obtain
		\[
		V_{pp}^{-1}f=\sum_{n,m}\frac{f_{nm}}{\lambda_n^{(p)}}b_{p,nm}.
		\]
		Equation \eqref{eq:lambda-asymptotic} gives constants $N_p$ and $C_p$ for which
		$\abs{1/\lambda_n^{(p)}}\leq C_p\langle n\rangle$ whenever $n\geq N_p$.
		The finitely many remaining eigenvalues are nonzero by
	Assumption~\ref{ass:nonresonance}. Enlarging $C_p$ gives the same
		inequality for all $n$. Therefore
		\begin{align*}
			\norm{V_{pp}^{-1}f}_{H^{-1/2}(\Gamma_p)}^2
			&=\sum_{n,m}\langle n\rangle^{-1}
			\frac{\abs{f_{nm}}^2}{\abs{\lambda_n^{(p)}}^2}\\
			&\leq C_p^2\sum_{n,m}\langle n\rangle\abs{f_{nm}}^2
			= C_p^2\norm{f}_{H^{1/2}(\Gamma_p)}^2.
		\end{align*}
	The extension to arbitrary $t\in\R$ follows by the same argument.
	\end{proof}
	
	In contrast to the self-interaction blocks, all cross-interaction blocks are smooth.
	\begin{lemma}\label{lem:cross-smoothing}
		Under Assumption \ref{ass:nonresonance}, for $p\ne q$ and arbitrary $t\in\R$,
		\[
		V_{pq}:H^{-1/2}(\Gamma_q)\longrightarrow H^t(\Gamma_p),
		\qquad
		A_{pq}:H^{-1/2}(\Gamma_q)\longrightarrow H^t(\Gamma_p)
		\]
		are bounded. In particular, $A_{pq}$ is compact on $H^{-1/2}(\Gamma_q)$.
	\end{lemma}
	
	\begin{proof}
		For $p\ne q$, strict separation implies
		$\abs{\bm{x}-\bm{y}}\geq d_{pq}-a_p-a_q>0$ for
		$(\bm{x},\bm{y})\in\Gamma_p\times\Gamma_q$, which means the kernel $\Phi_k(\bm{x},\bm{y})$ is smooth. An integral operator with a smooth kernel maps every distributional Sobolev space on $\Gamma_q$ into every Sobolev space on $\Gamma_p$. This proves the assertion for $V_{pq}$. The smoothing property for $A_{pq}$ follows from the definition $A_{pq}=V_{pp}^{-1}V_{pq}$ and Lemma \ref{lem:self-inverse}. 
	\end{proof}
	
	Assumption~\ref{ass:nonresonance}, together with the uniqueness of the
	exterior scattering problem~\cite{ColtonKress2019}, implies that
	\begin{equation}\label{eq:V-isomorphism}
		V:\bigoplus_{p=1}^MH^{-1/2}(\Gamma_p)
		\longrightarrow
		\bigoplus_{p=1}^MH^{1/2}(\Gamma_p)
	\end{equation}
	is an isomorphism. By Lemma~\ref{lem:self-inverse}, \(\bD^{-1}\) is also an
	isomorphism between the corresponding coefficient spaces. Consequently, we have the following proposition.
	\begin{proposition}\label{prop:W-invertible}
		Under Assumption~\ref{ass:nonresonance}, the operator
		\begin{equation}\label{eq:W-invertible}
			\bW=\bD^{-1}V=\Id+\bA:\seqsp^{-1/2}\longrightarrow\seqsp^{-1/2}
		\end{equation}
		is boundedly invertible.
	\end{proposition}

	Although the MEM can compute the multiple scattering solution exactly, in practice, the infinite expansion must be truncated.
	Let $P_{p,N}$ retain the modes $0\leq n\leq N$ on sphere $p$, put
	$Q_{p,N}:=I-P_{p,N}$, and set
	\begin{equation}\label{eq:projectors}
		\PN:=\operatorname{diag}(P_{1,N},\ldots,P_{M,N}),
		\qquad
		\QN:=\Id-\PN.
	\end{equation}
	The truncated space has dimension $M(N+1)^2$. The degree-$N$ MEM seeks $\Phi_N\in\operatorname{Ran}\PN$ such that
	\begin{equation}\label{eq:finite-system}
		\PN(\Id+\bA)\PN\Phi_N=\PN G.
	\end{equation}
	 Define the operator in $ \seqsp^{-1/2}$:
	\begin{equation}\label{eq:WN}
		\bW_N:=\Id+\PN\bA\PN.
	\end{equation}
	Then \eqref{eq:finite-system} is equivalent to
	$\bW_N\Phi_N=\PN G$ together with $\QN\Phi_N=0$.
	
	In the end, the convergence problem is to determine whether the solution $\Phi_N$ of the truncated system \eqref{eq:finite-system} converges to the solution $\Phi$ of the full system \eqref{eq:infinite-system} and to quantify the error $\norm{\Phi-\Phi_N}_{\seqsp^{-1/2}}$ in terms of the truncation order $N$.
	
	
	\section{Mathematical apparatus  }\label{sec:preparations}
	
	This section develops the mathematical tools needed to analyze the truncation error. In particular, Proposition~\ref{prop:error-identity} reduces the convergence analysis to the estimates of two tails: the discarded incident data tail $\QN G$ and the discarded interaction operator tail $\bA-\PN\bA\PN$. The remainder of this section establishes the required special-function estimates and, in particular, shows that angular multiplicities, Sobolev weights, and derivatives contribute only polynomial prefactors.
   
   In the following, $C$ denotes a positive constant that may depend on the fixed wavenumber, the sphere configuration, and the fixed Sobolev indices under consideration, but is independent of all angular degrees and of the truncation parameter $N$. Its value may change from one occurrence to the next.

	
	\begin{proposition}\label{prop:error-identity}
		Let $N$ be sufficiently large such that $\bW_N$ is invertible on $\seqsp^{-1/2}$,
		and let $\Phi_N$ solve \eqref{eq:finite-system}. Extend $\Phi_N$ by zero
		on $\QN(\seqsp^{-1/2})$. Then
		\begin{equation}\label{eq:error-identity}
			\Phi-\Phi_N
			=\bW_N^{-1}
			\left[
			\QN G-(\bA-\PN\bA\PN)\Phi
			\right].
		\end{equation}
		Consequently, 
		\begin{eqnarray}\label{eq:abstract-error-bound}
			&&\norm{\Phi-\Phi_N}_{ \seqsp^{-1/2}}\notag\\
			\leq
			&&\norm{\bW_N^{-1}}_{\seqsp^{-1/2}\to\seqsp^{-1/2}}
			\left(
			\norm{\QN G}_{ \seqsp^{-1/2}}
			+\norm{\bA-\PN\bA\PN}_{\seqsp^{-1/2}\to\seqsp^{-1/2}}
			\norm{\Phi}_{ \seqsp^{-1/2}}
			\right).
		\end{eqnarray}
	\end{proposition}
	
	\begin{proof}
		The infinite and truncated systems are
		$\bW\Phi=G$ and $\bW_N\Phi_N=\PN G$, respectively. Therefore, it holds
		\begin{align*}
			\bW_N(\Phi-\Phi_N)
			&=\Phi+\PN\bA\PN\Phi-\PN G\\
			&=G-\bA\Phi+\PN\bA\PN\Phi-\PN G\\
			&=\QN G-(\bA-\PN\bA\PN)\Phi.
		\end{align*}
		Left multiplication by $\bW_N^{-1}$ proves
		\eqref{eq:error-identity}. Taking norms gives equation
		\eqref{eq:abstract-error-bound}.
	\end{proof}
	
	The invertibility of $\bW_N$ for sufficiently large $N$ will be
	established in Lemma~\ref{lem:stability}, once the interaction-tail estimate
	is available. The exact identity~\eqref{eq:error-identity} shows that the proof
	of the full convergence theorem requires separate estimates of $\QN G$ and
	$\bA-\PN\bA\PN$. The following technical lemma will be used repeatedly in
	deriving estimates for these two tails.

	\begin{lemma}\label{lem:poly-tail}
		Let $0<\theta<1$ and $\alpha\in\R$. Then
		\begin{equation}\label{eq:poly-tail}
			\lim_{N\to\infty}
			\left(\sum_{n>N}\langle n\rangle^{\alpha}\theta^{2n}\right)^{1/(2N)}
			=\theta.
		\end{equation}
		Moreover, for every $\eta$ satisfying $\theta<\eta<1$, there exists a constant $C>0$ such that
		\begin{equation}\label{bbou}
			\left(
			\sum_{n>N} \langle n\rangle^{\alpha}\theta^{2n}
			\right)^{1/2}
			\leq
			C\eta^N
		\end{equation}
		for every $N\geq1$.
	\end{lemma}
	
	\begin{proof}
		Fix $\eta$ with $\theta<\eta<1$. Since exponential growth dominates polynomial growth, there exists $N_0\geq1$ such that
		$\langle n\rangle^{\alpha}\leq(\eta/\theta)^{2n}$ for all $n\geq N_0$. Thus, for $N\geq N_0$,
		\[
		\sum_{n>N}\langle n\rangle^{\alpha}\theta^{2n}
		\leq\sum_{n>N}\eta^{2n}
		=\frac{\eta^{2(N+1)}}{1-\eta^2}.
		\]
		After enlarging the constant to cover the finitely many indices $1\leq N<N_0$, this proves \eqref{bbou}. Taking $2N$-th roots gives
		\[
		\limsup_{N\to\infty}
		\left(\sum_{n>N}\langle n\rangle^{\alpha}\theta^{2n}\right)^{1/(2N)}
		\leq\eta.
		\]
		Letting $\eta\downarrow\theta$ yields the upper bound $\theta$. Conversely, retaining only the term with index $N+1$ gives
		\[
		\left(\sum_{n>N}\langle n\rangle^{\alpha}\theta^{2n}\right)^{1/(2N)}
		\geq
		\langle N+1\rangle^{\alpha/(2N)}\theta^{1+1/N},
		\]
		whose limit is $\theta$. This proves \eqref{eq:poly-tail}.
	\end{proof}

	Next, we recall the large-order behavior of the Bessel functions and their derivatives.
 The asymptotics \eqref{eq:j-asymptotic}--\eqref{eq:h-asymptotic} imply that, for fixed $0<a<R$, 
	\begin{align}
		\lim_{n\to\infty}
		\abs{\frac{h_n^{(1)}(kR)}{h_n^{(1)}(ka)}}^{1/n}
		&=\frac{a}{R},\label{eq:h-ratio-root}\\
		\lim_{n\to\infty}
		\abs{j_n(ka)h_n^{(1)}(kR)}^{1/n}
		&=\frac{a}{R}.\label{eq:jh-root}
	\end{align}
	The convergence is uniform when $a$ and $R$ range over compact sets satisfying $0<a<R$. In addition, the derivatives satisfy
	\begin{equation}\label{eq:h-derivative-recurrence}
		\frac{\dd}{\dd z}h_n^{(1)}(z)
		=h_{n-1}^{(1)}(z)-\frac{n+1}{z}h_n^{(1)}(z).
	\end{equation}
	For a given sphere $\Gamma_q$, let $r=\abs{\bm{x}-\bm{c}_q}$ and $\widehat{\bm r}=(\bm{x}-\bm{c}_q)/r$. Since the radial and tangential parts are orthogonal, it holds
	\[
	\nabla\bigl(f(r)Y_n^m(\widehat{\bm r})\bigr)
	=f'(r)Y_n^m(\widehat{\bm r})\widehat{\bm r}
	+\frac{f(r)}{r}\nabla_{\Sph}Y_n^m(\widehat{\bm r}).
	\]
	Combining this identity with
	\eqref{eq:value-shell}--\eqref{eq:gradient-shell} in the Appendix yields
	\begin{align}
		\sum_{m=-n}^n
		\abs{h_n^{(1)}(kr)Y_n^m(\widehat{\bm r})}^2
		&=\frac{2n+1}{4\pi}\abs{h_n^{(1)}(kr)}^2,
		\label{eq:h-value-shell}\\
		\sum_{m=-n}^n
		\abs{\nabla\bigl(h_n^{(1)}(kr)Y_n^m(\widehat{\bm r})\bigr)}^2
		&=\frac{2n+1}{4\pi}
		\left(
		k^2\abs{h_n^{(1)\prime}(kr)}^2
		+\frac{n(n+1)}{r^2}\abs{h_n^{(1)}(kr)}^2
		\right).
		\label{eq:h-gradient-shell}
	\end{align}
	The same formulas hold with $h_n^{(1)}$ replaced by $j_n$.
	
	The following lemma provides the fundamental estimates for the convergence analysis. It
	converts the sum over all orders $m$ in a degree-\(n\) multipole family into physical value and gradient energies,
	so that only the estimate with respect to the degree $n$ remains.
	
	\begin{lemma}\label{lem:differentiated-shell}
		   Let $\bm{c}_q$ and $a_q$ denote the center and radius of sphere $\Gamma_q$, respectively, and set $K:=\Gamma_p$ for some $p\ne q$. For $\bm{x}\in K$, define
		\[
			r_q(\bm{x}):=\abs{\bm{x}-\bm{c}_q},
			\qquad
			\widehat{\bm r}_q(\bm{x}):=\frac{\bm{x}-\bm{c}_q}{r_q(\bm{x})}.
		\]
		Assume that, for all $\bm{x}\in K$,
		\[
			R_-\leq r_q(\bm{x})\leq R_+,
			\qquad R_->a_q.
		\]
		For every $\eta$ satisfying $a_q/R_-<\eta<1$, there exist constants $C>0$ and
		$\mu\geq0$, independent of $n$, such that for all $n\ge 1$,
		\begin{equation}\label{eq:differentiated-shell}
			\sum_{m=-n}^n
			\norm{j_n(ka_q)h_n^{(1)}(kr_q)
				Y_n^m(\widehat{\bm r}_q)}_{H^{1/2}(K)}^2
			\leq C\langle n\rangle^\mu\eta^{2n},
		\end{equation}
		and
		\begin{equation}\label{eq:differentiated-ratio}
			\sum_{m=-n}^n
			\norm{
				\frac{h_n^{(1)}(kr_q)}{h_n^{(1)}(ka_q)}
				Y_n^m(\widehat{\bm r}_q)
			}_{H^{1/2}(K)}^2
			\leq C\langle n\rangle^\mu\eta^{2n}.
		\end{equation}
	\end{lemma}
	
	\begin{proof}
		Fix
		\[
			\frac{a_q}{R_-}<\theta<\eta.
		\]
		The uniform large-order estimate \eqref{eq:jh-root} and the recurrence
		\eqref{eq:h-derivative-recurrence} imply that, for some $L\geq0$, \(n\ge N_0,\)
		\begin{align}
			\sup_{R_-\leq r\leq R_+}
			\abs{j_n(ka_q)h_n^{(1)}(kr)}
			&\leq C\langle n\rangle^L\theta^n,
			\label{eq:degree-value-uniform}\\
			\sup_{R_-\leq r\leq R_+}
			\abs{j_n(ka_q)k h_n^{(1)\prime}(kr)}
			&\leq C\langle n\rangle^{L+1}\theta^n.
			\label{eq:degree-gradient-uniform}
		\end{align}
        Enlarging \(C\), if necessary, absorbs the finitely many degrees below the large-order threshold, so the estimates hold for all \(n\ge1\).
		Set
		\[
			u_{nm}(\bm{x}):=j_n(ka_q)h_n^{(1)}(kr_q(\bm{x}))
			Y_n^m(\widehat{\bm r}_q(\bm{x})),
			\qquad \bm{x}\in K.
		\]
		By equation \eqref{eq:h-value-shell}, we have
		\begin{align}
			\sum_{m=-n}^n\abs{u_{nm}(\bm{x})}^2
			&=\frac{2n+1}{4\pi}
			\abs{j_n(ka_q)h_n^{(1)}(kr_q(\bm{x}))}^2\notag\\
			&\leq C\langle n\rangle^{2L+1}\theta^{2n}.
			\label{eq:degree-value-pointwise}
		\end{align}
		Likewise, equation \eqref{eq:h-gradient-shell} gives
		\begin{align}
			\sum_{m=-n}^n\abs{\nabla u_{nm}(\bm{x})}^2
			&=\abs{j_n(ka_q)}^2\frac{2n+1}{4\pi}
			\left[
			k^2\abs{h_n^{(1)\prime}(kr_q(\bm{x}))}^2
			+\frac{n(n+1)}{r_q(\bm{x})^2}
			\abs{h_n^{(1)}(kr_q(\bm{x}))}^2
			\right]\notag\\
			&\leq C\langle n\rangle^{2L+3}\theta^{2n},
			\qquad \bm{x}\in K.
			\label{eq:degree-gradient-pointwise}
		\end{align}
		Since $\abs{\nabla_K u_{nm}}\leq\abs{\nabla u_{nm}}$, where $\nabla_K$
		denotes the tangential gradient on $K$, integration over $K$ yields
		\begin{align*}
			\sum_{m=-n}^n\norm{u_{nm}}_{H^1(K)}^2
			&\leq C\sup_{\bm{x}\in K}
			\sum_{m=-n}^n
			\left(\abs{u_{nm}(\bm{x})}^2+\abs{\nabla u_{nm}(\bm{x})}^2\right)\\
			&\leq C\langle n\rangle^{2L+3}\theta^{2n}
			\leq C\langle n\rangle^{\mu}\eta^{2n}.
		\end{align*}
		The continuous embedding $H^1(K)\hookrightarrow H^{1/2}(K)$ then proves
		\eqref{eq:differentiated-shell}.
		
		The estimate \eqref{eq:differentiated-ratio} follows by the same argument.
		Indeed, the uniform ratio estimate \eqref{eq:h-ratio-root}, together with
		\eqref{eq:h-derivative-recurrence}, gives
		\begin{align*}
			\sup_{R_-\leq r\leq R_+}
			\abs{\frac{h_n^{(1)}(kr)}{h_n^{(1)}(ka_q)}}
			&\leq C\langle n\rangle^L\theta^n,\\
			\sup_{R_-\leq r\leq R_+}
			\abs{\frac{k h_n^{(1)\prime}(kr)}{h_n^{(1)}(ka_q)}}
			&\leq C\langle n\rangle^{L+1}\theta^n.
		\end{align*}
	\end{proof}
	
	The geometric meaning of Lemma~\ref{lem:differentiated-shell} is clear.
		At each degree $n$, the addition identities (see Appendix \ref{app:spherical-identities}) first sum all $2n+1$
		orders exactly. Separation of the centers then contributes the
		geometric ratio $a_q/R_-$, whereas
		order multiplicity, derivatives, and Sobolev weights contribute
		only powers of $n$. Thus the lemma turns the coupled degree and order
		translation problem into a single large-degree estimate without ever
		bounding individual Gaunt coefficients. This observation is central to
		the convergence analysis below.

	\section{Full convergence analysis}\label{sec:full-analysis}
	
	By Proposition~\ref{prop:error-identity}, the error decomposes into two
	distinct contributions. In this section, we first estimate the truncation
	tail of the preconditioned interaction operator $\bA-\PN\bA\PN$ and then the incident-data
	tail $\QN G$. Combining these estimates with uniform stability yields the full convergence theorem.
	
	In order to estimate the interaction truncation error $\bA-\PN\bA\PN$, we split it into target-side and source-side high-degree tails and bound them separately. We begin with the kernel representation of the interaction block. For $p\ne q$, denote the directed center-to-surface distances
	\begin{equation}\label{eq:directed-distances}
		R_{pq}^{\rm tar}:=d_{pq}-a_q,
		\qquad
		R_{pq}^{\rm src}:=d_{pq}-a_p.
	\end{equation}
	Strict disjointness gives $R_{pq}^{\rm tar}>a_p$ and
	$R_{pq}^{\rm src}>a_q$. 
	
	\begin{proposition}
		\label{prop:interaction-kernel}
		Under Assumption \ref{ass:nonresonance},  choose any two spheres $\Gamma_p$ and $\Gamma_q$ with $p\neq q$.  For $\bm{y}\in\Gamma_q$, define
		\[
		R_p(\bm{y}):=|\bm{y}-\bm{c}_p|,
		\qquad
		\widehat{\bm R}_p(\bm{y}):=\frac{\bm{y}-\bm{c}_p}{R_p(\bm{y})}.
		\]
		Then, for every $\varphi_q\in H^{-1/2}(\Gamma_q)$, the $(n,m)$-th
		coefficient of $A_{pq}\varphi_q$, where
		$A_{pq}:=V_{pp}^{-1}V_{pq}$, with respect to the normalized basis
		\[
		b_{p,nm}(\bm{x})
		=
		\frac1{a_p}
		Y_n^m\!\left(\widehat{\bm{x}-\bm{c}_p}\right)
		\]
		is
		\[
		(A_{pq}\varphi_q)_{nm}
		=
		\int_{\Gamma_q}
		F_{p,nm}(\bm{y})\varphi_q(\bm{y})\,\dd s_{\bm y},
		\]
		where
		\begin{equation}\label{eq:A-kernel}
			F_{p,nm}(\bm{y})
			=
			\frac{
				h_n^{(1)}(kR_p(\bm{y}))
			}{
				a_p h_n^{(1)}(ka_p)
			}
			\overline{
				Y_n^m\!\left(\widehat{\bm R}_p(\bm{y})\right)
			}.
		\end{equation}
	\end{proposition}
	
	\begin{proof}
		Since $\bm{y}\in\Gamma_q$ and the spheres are strictly disjoint, we have
		\[
		d_{pq}+a_q\ge R_p(\bm{y})\ge d_{pq}-a_q>a_p.
		\]
		Hence, for $\bm{x}\in\Gamma_p$, the addition formula for the Helmholtz
		fundamental solution gives
		\[
		\Phi_k(\bm{x},\bm{y})
		=
		ik
		\sum_{\ell=0}^{\infty}
		\sum_{s=-\ell}^{\ell}
		j_\ell(ka_p)
		h_\ell^{(1)}(kR_p(\bm{y}))
		Y_\ell^s\!\left(\widehat{\bm{x}-\bm{c}_p}\right)
		\overline{
			Y_\ell^s\!\left(\widehat{\bm R}_p(\bm{y})\right)
		}.
		\]
		Projecting $V_{pq}\varphi_q$ onto $b_{p,nm}$ and using
		\[
		\dd s_{\bm x}=a_p^2\,\dd\widehat{\bm x},
		\qquad
		\int_{\mathbb S^2}
		Y_\ell^s(\widehat{\bm x})
		\overline{Y_n^m(\widehat{\bm x})}
		\,\dd\widehat{\bm x}
		=
		\delta_{\ell n}\delta_{sm},
		\]
		yields
		\[
		(V_{pq}\varphi_q)_{nm}
		=
		ik\,a_p j_n(ka_p)
		\int_{\Gamma_q}
		h_n^{(1)}(kR_p(\bm{y}))
		\overline{
			Y_n^m\!\left(\widehat{\bm R}_p(\bm{y})\right)
		}
		\varphi_q(\bm{y})\,\dd s_{\bm y}.
		\]
		On the other hand, we have
		\[
		V_{pp}b_{p,nm}
		=
		\lambda_n^{(p)}b_{p,nm},
		\qquad
		\lambda_n^{(p)}
		=
		ik\,a_p^2
		j_n(ka_p)h_n^{(1)}(ka_p).
		\]
		Therefore, it holds
		\[
		(A_{pq}\varphi_q)_{nm}
		=
		\frac{(V_{pq}\varphi_q)_{nm}}
		{\lambda_n^{(p)}},
		\]
		and cancellation of $ik\,a_pj_n(ka_p)$ gives
		\[
		(A_{pq}\varphi_q)_{nm}
		=
		\int_{\Gamma_q}
		\frac{
			h_n^{(1)}(kR_p(\bm{y}))
		}{
			a_p h_n^{(1)}(ka_p)
		}
		\overline{
			Y_n^m\!\left(\widehat{\bm R}_p(\bm{y})\right)
		}
		\varphi_q(\bm{y})\,\dd s_{\bm y}.
		\]
	\end{proof}
	
	\begin{remark}
		A key analytic ingredient in the two-dimensional convergence proof given by \cite{Fitzpatrick2021} is a direct estimate
		of the translation coefficients. The lack of a tractable three-dimensional
		analogue is a principal obstacle to extending that argument.
		Proposition~\ref{prop:interaction-kernel} provides a substitute for such
		coefficient-wise estimates: all Gaunt couplings are resummed into the smooth
		kernel $F_{p,nm}$, after which the sum over the order $m$ is handled by the
		spherical harmonic addition theorem. A detailed comparison with the
		two-dimensional argument is given in Section~\ref{sec:first}.
	\end{remark}
	\begin{lemma}[Target-side high-degree interaction tail]\label{lem:target-tail}
		For $p\ne q$, it holds
		\begin{equation}\label{eq:target-tail-root}
			\limsup_{N\to\infty}
			\norm{Q_{p,N}A_{pq}}_{H^{-1/2}(\Gamma_q)\to H^{-1/2}(\Gamma_p)}^{1/N}
			\leq\frac{a_p}{d_{pq}-a_q}.
		\end{equation}
	\end{lemma}
	
	\begin{proof}
		The vectors
		\[
			e_{q,\ell r}:=\langle \ell\rangle^{1/2}b_{q,\ell r},
			\qquad \ell\geq0,\quad -\ell\leq r\leq\ell,
		\]
		form an orthonormal basis of $H^{-1/2}(\Gamma_q)$. For an operator
		$T:H^{-1/2}(\Gamma_q)\to H^{-1/2}(\Gamma_p)$, define the weighted
		Hilbert--Schmidt norm by
		\[
			\norm{T}_{\HS(-1/2\to-1/2)}^2
			:=\sum_{\ell=0}^{\infty}\sum_{r=-\ell}^{\ell}
			\norm{Te_{q,\ell r}}_{H^{-1/2}(\Gamma_p)}^2.
		\]
		This definition is independent of the chosen orthonormal basis, and the
		Hilbert--Schmidt norm dominates the operator norm \cite{Conway1990}:
		\[
			\norm{T}_{H^{-1/2}(\Gamma_q)\to H^{-1/2}(\Gamma_p)}
			\leq \norm{T}_{\HS(-1/2\to-1/2)}.
		\]
		We prove the stronger estimate in this norm. By equation
		\eqref{eq:A-kernel}, Tonelli's theorem \cite{Folland1999}, and Parseval's identity on
		$\Gamma_q$, we obtain
		\begin{align}
			\norm{Q_{p,N}A_{pq}}_{\HS(-1/2\to-1/2)}^2
			&=\sum_{\ell=0}^{\infty}\sum_{r=-\ell}^{\ell}
			\norm{Q_{p,N}A_{pq}e_{q,\ell r}}_{H^{-1/2}(\Gamma_p)}^2\notag\\
			&=\sum_{\ell=0}^{\infty}\sum_{r=-\ell}^{\ell}
			\langle\ell\rangle
			\sum_{n>N}\sum_{m=-n}^{n}\langle n\rangle^{-1}
			\abs{(A_{pq}b_{q,\ell r})_{nm}}^2\notag\\
			&=\sum_{n>N}\sum_{m=-n}^{n}\langle n\rangle^{-1}
			\sum_{\ell=0}^{\infty}\sum_{r=-\ell}^{\ell}\langle\ell\rangle
			\abs{\int_{\Gamma_q}F_{p,nm}(\bm{y})b_{q,\ell r}(\bm{y})\dd  s_{\bm y}}^2\notag\\
			&=\sum_{n>N}\sum_{m=-n}^n
			\langle n\rangle^{-1}
			\norm{F_{p,nm}}_{H^{1/2}(\Gamma_q)}^2.
			\label{eq:target-weighted-HS}
		\end{align}
		Note that the integral in \eqref{eq:A-kernel} is bilinear rather than
		sesquilinear. For the complex basis in \eqref{eq:Y-definition}, the identity
		$\overline{Y_\ell^r}=(-1)^rY_\ell^{-r}$ merely permutes the orders
		in the Parseval sum, which justifies the last equality in \eqref{eq:target-weighted-HS}.
		
		For $\bm{y}\in\Gamma_q$,
		$R_p(\bm{y})\geq d_{pq}-a_q=R_{pq}^{\rm tar}$. Applying the second part of
		Lemma \ref{lem:differentiated-shell} to the kernel in \eqref{eq:A-kernel}, we obtain that, for every
		$a_p/R_{pq}^{\rm tar}<\eta<1$, there exist $C>0$ and $\mu>0$, such that
		\[
		\sum_{m=-n}^n\norm{F_{p,nm}}_{H^{1/2}(\Gamma_q)}^2
		\leq C\langle n\rangle^{\mu}\eta^{2n}.
		\]
		Insert this estimate into \eqref{eq:target-weighted-HS}. The additional factor
		$\langle n\rangle^{-1}$ provides only polynomial decay. By
		Lemma~\ref{lem:poly-tail}, for any $\eta<\theta<1$, it holds
		\[
		\norm{Q_{p,N}A_{pq}}_{\HS(-1/2\to-1/2)}
		\leq C\theta^N.
		\]
		Taking the $N$th root and \(N\to\infty\) gives
\[\limsup_{N\to\infty}\|Q_{p,N}A_{pq}\|_{\HS(-1/2\to-1/2)}^{1/N}\le\theta.\]
Thus letting $\theta\downarrow a_p/(d_{pq}-a_q)$ yields
		\eqref{eq:target-tail-root}.
	\end{proof}
	
	
	We next estimate the source-side high-degree tail. For a normalized source basis function, the single-layer field on $\Gamma_p$ is
	\begin{equation}\label{eq:source-mode}
		(V_{pq}b_{q,nm})(\bm{x})
		=\ii k a_qj_n(ka_q)h_n^{(1)}(kr_q(\bm{x}))
		Y_n^m(\widehat{\bm r}_q(\bm{x})),
		\qquad \bm{x}\in\Gamma_p,
	\end{equation}
	where $r_q(\bm{x})=\abs{\bm{x}-\bm{c}_q}\geq d_{pq}-a_p$.
	
	\begin{lemma}[Source-side high-degree interaction tail]\label{lem:source-tail}
		For $p\ne q$, it holds
		\begin{equation}\label{eq:source-tail-root}
			\limsup_{N\to\infty}
			\norm{A_{pq}Q_{q,N}}_{H^{-1/2}(\Gamma_q)\to H^{-1/2}(\Gamma_p)}^{1/N}
			\leq\frac{a_q}{d_{pq}-a_p}.
		\end{equation}
	\end{lemma}
	
	\begin{proof}
		Again it suffices to estimate the weighted Hilbert--Schmidt norm. 
		Since the vectors
		$e_{q,nm}=\langle n\rangle^{1/2}b_{q,nm}$ are orthonormal in
		$H^{-1/2}(\Gamma_q)$, it holds
		\begin{align}
			\norm{A_{pq}Q_{q,N}}_{\HS(-1/2\to-1/2)}^2
			&=\sum_{n>N}\sum_{m=-n}^n
			\norm{A_{pq}e_{q,nm}}_{H^{-1/2}(\Gamma_p)}^2\notag\\
			&=\sum_{n>N}\sum_{m=-n}^n\langle n\rangle
			\norm{A_{pq}b_{q,nm}}_{H^{-1/2}(\Gamma_p)}^2.
			\label{eq:source-weighted-HS}
		\end{align}
		By Lemma \ref{lem:self-inverse}, we have
		\begin{equation}\label{eq:source-self-inverse}
			\norm{A_{pq}b_{q,nm}}_{H^{-1/2}(\Gamma_p)}
			=\norm{V_{pp}^{-1}V_{pq}b_{q,nm}}_{H^{-1/2}(\Gamma_p)}
			\leq C_p\norm{V_{pq}b_{q,nm}}_{H^{1/2}(\Gamma_p)}.
		\end{equation}
		This is the step at which the one-order shift caused by the inverse self block enters.
		
		Substitution of \eqref{eq:source-mode} followed by the degree-wise identities
		\eqref{eq:h-value-shell}--\eqref{eq:h-gradient-shell} gives
		\begin{align}
			& \sum_{m=-n}^n\norm{V_{pq}b_{q,nm}}_{H^{1/2}(\Gamma_p)}^2\notag\\
			&\quad\leq C\abs{j_n(ka_q)}^2
			\int_{\Gamma_p}\frac{2n+1}{4\pi}
			\left[
			\abs{h_n^{(1)}(kr_q)}^2
			+k^2\abs{h_n^{(1)\prime}(kr_q)}^2
			+\frac{n(n+1)}{r_q^2}\abs{h_n^{(1)}(kr_q)}^2
			\right]\dd s_{\bm x}.
			\label{eq:explicit-gradient-energy}
		\end{align}
		Since $r_q\geq R_{pq}^{\rm src}:=d_{pq}-a_p$, equations
		\eqref{eq:jh-root} and \eqref{eq:h-derivative-recurrence} show that, for every
		$a_q/R_{pq}^{\rm src}<\theta<1$, the right-hand side of
		\eqref{eq:explicit-gradient-energy} is at most
		$C\langle n\rangle^{\mu}\theta^{2n}$ for some $C>0$ and $\mu>0$.
		Consequently, equation \eqref{eq:source-self-inverse} and the preceding estimate give
		\[
		\sum_{m=-n}^n
		\norm{A_{pq}b_{q,nm}}_{H^{-1/2}(\Gamma_p)}^2
		\leq C_\theta\langle n\rangle^{\mu}\theta^{2n}.
		\]
		Insert this estimate  into equation
		\eqref{eq:source-weighted-HS}. Applying Lemma~\ref{lem:poly-tail}, taking
		the $N$th root, and then letting
		$\theta\downarrow a_q/(d_{pq}-a_p)$ yields
		\eqref{eq:source-tail-root}.
	\end{proof}
	
	We now combine the preceding estimates to bound the complete interaction tail.
	Define
	\begin{equation}\label{eq:rho-geo}
		\rho_{\rm geo}:=\max_{p\ne q}\frac{a_p}{d_{pq}-a_q}.
	\end{equation}
	Strict disjointness implies $\rho_{\rm geo}<1$.
	
	\begin{lemma}[Truncated interaction-operator tail]\label{lem:operator-tail}
		The following root bound holds:
		\begin{equation}\label{eq:operator-tail-root}
			\limsup_{N\to\infty}
			\norm{\bA-\PN\bA\PN}_{ \seqsp^{-1/2}\to \seqsp^{-1/2}}^{1/N}
			\leq\rho_{\rm geo}.
		\end{equation}
	\end{lemma}
	
	\begin{proof}
		Use the exact decomposition
		\begin{equation}\label{eq:tail-decomposition}
			\bA-\PN\bA\PN=\QN\bA+\PN\bA\QN.
		\end{equation}
		There are only finitely many off-diagonal blocks. Apply
		Lemma \ref{lem:target-tail} to every block of $\QN\bA$ and
		Lemma \ref{lem:source-tail} to every block of $\bA\QN$. The norm of a finite block matrix is bounded by a fixed finite multiple of the largest block norm, and such fixed constants disappear after taking the $N$th root. The maximum of the two directed root factors is $\rho_{\rm geo}$, which proves \eqref{eq:operator-tail-root}.
	\end{proof}
	
	Now let us study the error introduced by the truncation of incident data. 
	Let $f_p=u^{\inc}|_{\Gamma_p}$ and expand
	$f_p=\sum f_{p,nm}b_{p,nm}$. By equations \eqref{eq:G-definition} and
	\eqref{eq:self-eigenvalue}, we have
	\begin{equation}\label{eq:G-coeff-general}
		G_{p,nm}=-\frac{f_{p,nm}}{\lambda_n^{(p)}}.
	\end{equation}
	
	\begin{lemma}\label{lem:data-tail}
		The incident data tail satisfies the following estimates:
		\begin{enumerate}
			\item For plane-wave incidence, it holds
			\begin{equation}\label{eq:plane-data-tail}
				\limsup_{N\to\infty}\norm{\QN G}_{ \seqsp^{-1/2}}^{1/N}=0.
			\end{equation}
			\item For incidence generated by a point source $\bm{x}_0$ satisfying
			$d_{p0}:=\abs{\bm{x}_0-\bm{c}_p}>a_p$ for every $p$, it holds 
			\begin{equation}\label{eq:point-data-tail}
				\limsup_{N\to\infty}\norm{\QN G}_{ \seqsp^{-1/2}}^{1/N}
				\leq\rho_{\rm src}:=\max_p\frac{a_p}{d_{p0}}.
			\end{equation}
		\end{enumerate}
	\end{lemma}
	
	\begin{proof}
		For a plane wave, the standard regular expansion about $\bm{c}_p$ is
		\[
		\e^{\ii k\widehat{\bm d}\cdot(\bm{c}_p+a_p\bm{\omega})}
		=4\pi\e^{\ii k\widehat{\bm d}\cdot \bm{c}_p}
		\sum_{n,m}\ii^nj_n(ka_p)Y_n^m(\bm{\omega})
		\overline{Y_n^m(\widehat{\bm d})}.
		\]
		Because $b_{p,nm}=a_p^{-1}Y_n^m$, its surface coefficient is
		\[
		f_{p,nm}=4\pi a_p\e^{\ii k\widehat{\bm d}\cdot \bm{c}_p}
		\ii^nj_n(ka_p)\overline{Y_n^m(\widehat{\bm d})}.
		\]
		Substituting this expression into equation \eqref{eq:G-coeff-general} cancels
		$j_n(ka_p)$ and gives
		\begin{equation}\label{eq:G-plane-coeff}
			G_{p,nm}
			=-\frac{4\pi\ii^{n-1}\e^{\ii k\widehat{\bm d}\cdot \bm{c}_p}}
			{ka_ph_n^{(1)}(ka_p)}
			\overline{Y_n^m(\widehat{\bm d})}.
		\end{equation}
		The reciprocal of $h_n^{(1)}(ka_p)$ decays factorially.  The conclusion  \eqref{eq:plane-data-tail} follows by summing
		$\abs{Y_n^m(\widehat{\bm d})}^2$ over $m$ by \eqref{eq:value-shell}, and multiplying by the weight $\langle n\rangle^{-1}$.
		
		For a point source, let
		$\widehat{\bm d}_{p0}:=(\bm{x}_0-\bm{c}_p)/d_{p0}$. Equation \eqref{eq:green-expansion} in the Appendix gives
		\[
		f_{p,nm}=\ii k a_pj_n(ka_p)h_n^{(1)}(kd_{p0})
		\overline{Y_n^m(\widehat{\bm d}_{p0})},
		\]
		which implies
		\begin{equation}\label{eq:G-point-coeff}
			G_{p,nm}
			=-\frac{h_n^{(1)}(kd_{p0})}
			{a_ph_n^{(1)}(ka_p)}
			\overline{Y_n^m(\widehat{\bm d}_{p0})}.
		\end{equation}
		After the summation over the order m, the degree-$n$ contribution to
		$\norm{G_p}_{H^{-1/2}(\Gamma_p)}^2$ is a polynomial factor times
		$\abs{h_n^{(1)}(kd_{p0})/h_n^{(1)}(ka_p)}^2$.
		Equation \eqref{eq:h-ratio-root} and Lemma \ref{lem:poly-tail} give the root factor $a_p/d_{p0}$. Taking the maximum over the finite set of spheres yields \eqref{eq:point-data-tail}.
	\end{proof}
	
	
	The interaction-tail estimate is now available, so the truncated systems can
	be shown to inherit the invertibility of the infinite system.
	
	\begin{lemma}\label{lem:stability}
		Under Assumption \ref{ass:nonresonance}, there exists $N_0$ such that
		$\bW_N$ is invertible on $ \seqsp^{-1/2}$ for all $N\geq N_0$, and
		\begin{equation}\label{eq:uniform-stability}
			\sup_{N\geq N_0}\norm{\bW_N^{-1}}_{ \seqsp^{-1/2}\to \seqsp^{-1/2}}
			\leq2\norm{\bW^{-1}}_{ \seqsp^{-1/2}\to \seqsp^{-1/2}}.
		\end{equation}
		Consequently, the truncated MEM system \eqref{eq:finite-system} is uniquely solvable for every $N\geq N_0$.
	\end{lemma}
	
	\begin{proof}
		By equations \eqref{eq:WN} and \eqref{eq:infinite-system}, it holds
		\[
		\bW_N-\bW=\PN\bA\PN-\bA.
		\]
		Lemma~\ref{lem:operator-tail} implies \(\|\bW_N-\bW\|_{ \seqsp^{-1/2}\to \seqsp^{-1/2}}\to0\). Choose some sufficiently large $N_0$ such that for all $N\geq N_0$
		\[
		\norm{\bW^{-1}(\bW_N-\bW)}_{ \seqsp^{-1/2}\to \seqsp^{-1/2}}\leq\frac12,
		\]
		and factor
		\[
		\bW_N=\bW\left[\Id+\bW^{-1}(\bW_N-\bW)\right].
		\]
		The bracket is invertible by its Neumann series and its inverse has norm at most $2$, which proves \eqref{eq:uniform-stability}. Since $\bW_N$ is  identity on $\QN (\seqsp^{-1/2})$, its invertibility on the full space is equivalent to the invertibility of the truncated matrix on $\PN (\seqsp^{-1/2})$.
	\end{proof}
	
	Define
	\begin{equation}\label{eq:rho-data}
		\rho_{\rm data}:=
		\begin{cases}
			0,&\text{for plane-wave incidence},\\[1mm]
			\displaystyle\max_p\frac{a_p}{d_{p0}},&\text{for point-source incidence}.
		\end{cases}
	\end{equation}
	We now establish the full convergence result of the truncated MEM system.

	\begin{theorem}\label{thm:main}
		 Under Assumption \ref{ass:nonresonance}, let
		$\Phi$ and $\Phi_N$ be the solutions of equations \eqref{eq:infinite-system} and 
		\eqref{eq:finite-system}, respectively. Then $\Phi_N$ exists uniquely for all sufficiently large $N$ and
		\begin{equation}\label{eq:main-root}
			\limsup_{N\to\infty}
			\norm{\Phi-\Phi_N}_{ \seqsp^{-1/2}}^{1/N}
			\leq\max\{\rho_{\rm data},\rho_{\rm geo}\}.
		\end{equation}
		Equivalently, for every $r$ satisfying
		$\max\{\rho_{\rm data},\rho_{\rm geo}\}<r<1$, there is a constant
		$C>0$, independent of $N$, such that
		\begin{equation}\label{eq:main-geometric}
			\norm{\Phi-\Phi_N}_{ \seqsp^{-1/2}}\leq Cr^N
		\end{equation}
		for all sufficiently large $N$.
	\end{theorem}
	
	\begin{proof}
		Existence, uniqueness, and the uniform inverse bound follow from
		Lemma \ref{lem:stability}. Based on equation \eqref{eq:abstract-error-bound}, the first term on its right has root rate at most $\rho_{\rm data}$ by
		Lemma \ref{lem:data-tail}. The operator factor in the second term has root rate at most $\rho_{\rm geo}$ by Lemma \ref{lem:operator-tail}. The exact solution $\Phi$ and the stability constant are independent of $N$, which proves the estimate \eqref{eq:main-root}, as well as the estimate \eqref{eq:main-geometric}.
	\end{proof}
	
	\begin{remark}[Near- and far-field consequences]\label{rem:field-convergence}
		Let $U\Subset\Omega^+$ be a bounded smooth domain staying a positive distance from all spheres.
		The single-layer field map from $\seqsp^{-1/2}$ to $H^t(U)$ is bounded for
		every fixed integer $t\geq0$, and the far-field map is bounded into
		$L^2(\Sph)$. Hence the near-field and far-field errors inherit the same
		root bound $\max\{\rho_{\rm data},\rho_{\rm geo}\}$ as
		\eqref{eq:main-root}. 
	\end{remark}
	
	\section{First-transfer theory}\label{sec:first}
	
	The full theorem treats all orders of multiple reflection through an
	operator norm tail. Additional analyticity of the isolated sphere response
	permits a sharper rate for the explicitly identified first-transfer part of
	the error. We derive that refinement first and then compare our proof
	 with the existing two-dimensional theory \cite{Fitzpatrick2021}.
	
	
	Write
	\begin{equation}\label{eq:Phi-split}
		\Phi=G+\Phi^{\diff},
		\qquad
		\Phi^{\diff}:=\Phi-G.
	\end{equation}
	Substituting \eqref{eq:Phi-split} into equation \eqref{eq:error-identity} gives
	\begin{equation}\label{eq:three-part-error}
		\Phi-\Phi_N
		=\bW_N^{-1}
		\left[
		\underbrace{\QN G}_{\text{incident data tail}}
		-\underbrace{(\bA-\PN\bA\PN)G}_{\text{first-transfer}}
		-\underbrace{(\bA-\PN\bA\PN)\Phi^{\diff}}_{\text{higher transfers}}
		\right].
	\end{equation}
	For all sufficiently large $N$, define
	\begin{align}
		\mathcal E_N^{(1)}
		&:=\norm{\bW_N^{-1}}_{\seqsp^{-1/2}\to\seqsp^{-1/2}}
		\left(
		\norm{\QN G}_{ \seqsp^{-1/2}}
		+\norm{(\bA-\PN\bA\PN)G}_{ \seqsp^{-1/2}}
		\right),\label{eq:E1}\\
		\mathcal E_N^{(\diff)}
		&:=\norm{\bW_N^{-1}}_{\seqsp^{-1/2}\to\seqsp^{-1/2}}
		\norm{(\bA-\PN\bA\PN)\Phi^{\diff}}_{ \seqsp^{-1/2}}.
		\label{eq:Ediff}
	\end{align}
	Then
	\begin{equation}\label{eq:error-E-split}
		\norm{\Phi-\Phi_N}_{ \seqsp^{-1/2}}
		\leq\mathcal E_N^{(1)}+\mathcal E_N^{(\diff)}.
	\end{equation}
	The next result provides a more refined analysis of $\mathcal E_N^{(1)}$ by
	estimating the tail of an analytic local field.
	
	
	\begin{lemma}\label{lem:analytic-local}
		Under Assumption \ref{ass:nonresonance}, let $v$ solve the Helmholtz equation in the ball $B(\bm{c}_p,R)$, where $R>a_p$, and define
		\begin{equation}\label{eq:gv}
			g_v:=V_{pp}^{-1}(v|_{\Gamma_p}).
		\end{equation}
		Then
		\begin{equation}\label{eq:analytic-local-root}
			\limsup_{N\to\infty}\norm{Q_{p,N}g_v}_{H^{-1/2}(\Gamma_p)}^{1/N}
			\leq\frac{a_p}{R}.
		\end{equation}
	\end{lemma}
	
	\begin{proof}
		Fix $R'$ with $a_p<R'<R$. Regular separation of variables gives
		\begin{equation}\label{eq:v-regular}
			v(\bm{c}_p+r\bm{\omega})=\sum_{n,m}\alpha_{nm}j_n(kr)Y_n^m(\bm{\omega}),
			\qquad 0\leq r\leq R'.
		\end{equation}
		The trace at $r=R'$ belongs to $L^2(\Sph)$, so Parseval identity gives
		\begin{equation}\label{eq:alpha-R}
			\sum_{n,m}\abs{\alpha_{nm}j_n(kR')}^2<\infty.
		\end{equation}
		For a fixed $kR'>0$, the large-order asymptotic formula for $j_n(kR')$ shows that it is
		nonzero for all sufficiently large $n$. For those degrees, the coefficient
		of $v$ on $\Gamma_p$ in the normalized basis $b_{p,nm}$ is
		$a_p\alpha_{nm}j_n(ka_p)$, and division by $\lambda_n^{(p)}$ gives
		\begin{equation}\label{eq:gv-coeff}
			(g_v)_{nm}
			=\frac{\alpha_{nm}}{\ii k a_ph_n^{(1)}(ka_p)}
			=\frac{\alpha_{nm}j_n(kR')}{\ii k a_pj_n(kR')h_n^{(1)}(ka_p)}.
		\end{equation}
		By the asymptotics \eqref{eq:j-asymptotic} and \eqref{eq:h-asymptotic}, it holds
		\begin{equation}\label{eq:inverse-jh-root}
			\lim_{n\to\infty}
			\abs{\frac{1}{j_n(kR')h_n^{(1)}(ka_p)}}^{1/n}
			=\frac{a_p}{R'}.
		\end{equation}

		For any $\eta$ with $a_p/R'<\eta<1$, when $N$ is sufficiently large, equations \eqref{eq:alpha-R} and
		\eqref{eq:inverse-jh-root} therefore yield
		\begin{align*}
			\norm{Q_{p,N}g_v}_{H^{-1/2}(\Gamma_p)}^2
			&\leq C_\eta\sum_{n>N}\sum_{m=-n}^n
			\langle n\rangle^{-1}\eta^{2n}
			\abs{\alpha_{nm}j_n(kR')}^2\\
			&\leq C_\eta\eta^{2N}
			\sum_{n,m}\abs{\alpha_{nm}j_n(kR')}^2.
		\end{align*}
		Taking square roots and then $N$th roots, followed by the limits
		$\eta\downarrow a_p/R'$ and $R'\uparrow R$, yields
		\eqref{eq:analytic-local-root}.
	\end{proof}
	
	
	Let
	\begin{equation}\label{eq:isolated-field}
		v_q^{(0)}:=\mathcal S_qG_q,
	\end{equation}
	where \(\mathcal S_q\) denotes the single-layer potential supported on \(\Gamma_q\). Because $V_{qq}G_q=-u^{\inc}|_{\Gamma_q}$, this is exactly the field scattered by sphere $q$ when all other spheres are absent. Moreover, we have
	\begin{equation}\label{eq:AG-isolated}
		A_{pq}G_q=V_{pp}^{-1}(v_q^{(0)}|_{\Gamma_p}).
	\end{equation}
	\begin{remark}
	    By use of equation \eqref{eq:green-expansion}, the exterior single-layer expansion can be obtained, that's letting 
    $\phi_q=\sum_{n,m}\Phi_{q,nm}b_{q,nm},$ then for $|x-c_q|>a_q$,
    \begin{equation}\label{eq:density-field}
 (\mathcal S_q\phi_q)(x)
 =\sum_{n,m}\ii k a_qj_n(ka_q)\Phi_{q,nm}
 h_n^{(1)}(kr_q)Y_n^m(\widehat r_q).
\end{equation}
	\end{remark}
	\begin{lemma}\label{lem:effective-radius}
		When $p\ne q$, it holds:
		\begin{enumerate}[label=(\roman*),leftmargin=2.2em]
			\item For plane wave incidence, the exterior field $v_q^{(0)}$ admits an analytic continuation in
			$\R^3\setminus\{\bm{c}_q\}$. Consequently it is analytic in every ball
			$B(\bm{c}_p,R)$ with $R<d_{pq}$.
			\item For a point source, put
			\begin{equation}\label{eq:rq-star}
				r_q^*:=\frac{a_q^2}{d_{q0}}<a_q.
			\end{equation}
			Then the exterior field $v_q^{(0)}$ admits an analytic continuation in $\{\bm{x}\in\R^3|\abs{\bm{x}-\bm{c}_q}>r_q^*\}$ and therefore in every ball
			$B(\bm{c}_p,R)$ with $R<d_{pq}-r_q^*$.
		\end{enumerate}
	\end{lemma}
	
	\begin{proof}
		For a plane wave, combining equation \eqref{eq:G-plane-coeff} with the expansion \eqref{eq:density-field} shows that the outgoing degree-$n$ coefficient contains
		$j_n(ka_q)/h_n^{(1)}(ka_q)$. For every fixed $r>0$, multiplication by
		$h_n^{(1)}(kr)$ leaves a superexponentially decaying sequence because the factor
		$j_n(ka_q)$ retains a double-factorial denominator. The normally convergent series defines a \(C^\infty\) Helmholtz solution for every $r>0$, thus the analyticity then follows from analytic elliptic regularity, which proves (i).
		
		For point source incidence\eqref{eq:G-point-coeff} with the expansion \eqref{eq:density-field} shows that, up to
		angular factors and constants independent of $n$, the degree-$n$ radial
		factor is
		\begin{equation}\label{eq:point-four-factor}
			j_n(ka_q)h_n^{(1)}(kd_{q0})
			\frac{h_n^{(1)}(kr)}{h_n^{(1)}(ka_q)}.
		\end{equation}
		Using asymptotics \eqref{eq:j-asymptotic}--\eqref{eq:h-asymptotic} for all four factors gives its asymptotic behavior 
		\begin{equation}\label{eq:point-four-factor-asymptotic}
			\frac{-\ii}{k d_{q0}r}
			\frac{a_q}{2n+1}
			\left(\frac{a_q^2}{d_{q0}r}\right)^n
			\bigl(1+o(1)\bigr).
		\end{equation}
		 Thus the series converge normally on compact subsets of
		$r>r_q^*=a_q^2/d_{q0}$ as a \(C^\infty\) Helmholtz solution, which proves (ii) following from the analytic elliptic regularity.
	\end{proof}
	We now state the first-transfer theorem, which provides an estimate for
	$\mathcal E_N^{(1)}$.
	Define
	\begin{align}
		\rho_{\rm first}^{\rm pw}
		&:=\max_{p\ne q}\frac{a_p}{d_{pq}},\label{eq:rho-first-pw}\\
		\rho_{\rm first}^{\rm pt}
		&:=\max\left\{
		\max_p\frac{a_p}{d_{p0}},
		\max_{p\ne q}
		\frac{a_pd_{q0}}{d_{pq}d_{q0}-a_q^2}
		\right\}.
		\label{eq:rho-first-pt}
	\end{align}
	
	\begin{theorem}\label{thm:first-transfer}
		Under the assumption of Theorem \ref{thm:main}, it holds:
		\begin{align}
			\limsup_{N\to\infty}(\mathcal E_N^{(1)})^{1/N}
			&\leq\rho_{\rm first}^{\rm pw},
			&&\text{for plane wave incidence};\label{eq:first-pw}\\
			\limsup_{N\to\infty}(\mathcal E_N^{(1)})^{1/N}
			&\leq\rho_{\rm first}^{\rm pt},
			&&\text{for point source incidence}.\label{eq:first-pt}
		\end{align}
	\end{theorem}
	
	\begin{proof}
		By \eqref{eq:E1} and Lemma~\ref{lem:stability}, the stability factor
		$\norm{\bW_N^{-1}}$ is uniformly bounded and therefore does not affect
		the root rate. It thus suffices to analyze the decomposition
		\begin{equation}\label{eq:AG-tail-split}
			(\bA-\PN\bA\PN)G
			=\QN\bA G+\PN\bA\QN G.
		\end{equation}
		For a plane wave, Lemma \ref{lem:data-tail} shows that $\QN G$ is superexponentially small. Boundedness of $\bA$ makes the second term in
		\eqref{eq:AG-tail-split} superexponentially small as well. Since $Q_{p,N}A_{pp}G_p=0$, we only need to consider the crossing blocks. For its first term, equation
		\eqref{eq:AG-isolated} and Lemma \ref{lem:effective-radius}(i) show that
		$A_{pq}G_q$ is obtained by applying $V_{pp}^{-1}$ to a Helmholtz field analytic in every ball $B(\bm{c}_p,R)$ with $R<d_{pq}$. Lemma~\ref{lem:analytic-local} therefore gives
		\[
		\limsup_{N\to\infty}\norm{Q_{p,N}A_{pq}G_q}_{H^{-1/2}(\Gamma_p)}^{1/N}
		\leq\frac{a_p}{d_{pq}}.
		\]
		Taking the maximum over the finite set of ordered pairs proves equation
		\eqref{eq:first-pw}.
		
		For a point source, the tail and the term
		$\PN\bA\QN G$ have root rate at most
		$\max_p a_p/d_{p0}$. By Lemma \ref{lem:effective-radius}(ii), the analytic ball for the isolated response from sphere $q$, viewed about $\bm{c}_p$, can have every radius
		\[
		R<d_{pq}-\frac{a_q^2}{d_{q0}}.
		\]
		Applying Lemma \ref{lem:analytic-local} and letting $R$ approach this limit gives
		\[
		\limsup_{N\to\infty}\norm{Q_{p,N}A_{pq}G_q}_{H^{-1/2}(\Gamma_p)}^{1/N}
		\leq
		\frac{a_p}{d_{pq}-a_q^2/d_{q0}}
		=\frac{a_pd_{q0}}{d_{pq}d_{q0}-a_q^2}.
		\]
		Taking the maximum with the data tail factor yields \eqref{eq:first-pt}.
	\end{proof}
	
		Note that Theorem~\ref{thm:first-transfer} controls only the first-transfer term
		$\mathcal E_N^{(1)}$, not the higher-transfer remainder
		$\mathcal E_N^{(\diff)}$. The latter contains all subsequent reflections
		and is bounded by the full geometric rate established in
		Lemma~\ref{lem:operator-tail}, but it need not inherit the sharper rate
		associated with the isolated sphere analyticity radius.

	We conclude this section by comparing our analysis with the two-dimensional
	theory in \cite{Fitzpatrick2021}. At the operator level, the two approaches share the same structure. In particular, diagonal preconditioning gives 
	\[
	(\Id+\bA)\Phi=G,
	\qquad
	A_{pq}=V_{pp}^{-1}V_{pq},
	\]
	and the truncation error is controlled by the data tail and
	\[
	\bA-\PN\bA\PN=\QN\bA+\PN\bA\QN.
	\]
	Consequently, the full geometric root factor has the same form in both
	dimensions,
	\[
	\rho_{\rm geo}=\max_{p\ne q}\frac{a_p}{d_{pq}-a_q},
	\]
	while Theorem~\ref{thm:first-transfer} gives sharper analyticity-based factors
	for the isolated first-transfer component in the present setting.
	
	The essential difference is the proof of the interaction-tail estimate. In two
	dimensions, \cite{Fitzpatrick2021} starts from explicit Bessel/Hankel
	translation coefficients and reduces coupled double-index sums to
	hypergeometric functions. A direct three-dimensional analogue would also have
	to resolve the $2n+1$ orders and the Gaunt coupling between
	$(n,m)$ and $(\ell,r)$. Termwise absolute value estimates obscure
	orthogonality and produce unwieldy degree--order sums, which is precisely the
	difficulty identified in \cite{Fitzpatrick2021} for a three-dimensional
	extension.
	
	Our target/source decomposition avoids this route. For the target-side
	high-degree tail $Q_{p,N}A_{pq}$, fixing the output mode and projecting the
	Green kernel cancels the regular Bessel factor and leaves the ratio
	${h_n^{(1)}(kR)}/{h_n^{(1)}(ka_p)}$. 
	For the source-side high-degree tail $A_{pq}Q_{q,N}$, fixing the input mode
	produces a physical translated degree-$n$ spherical-wave family, while
	\[
	V_{pp}^{-1}:H^{1/2}(\Gamma_p)\to H^{-1/2}(\Gamma_p),
	\]
	reduces the estimate to its value and gradient energies. In both cases the
	spherical harmonic addition theorem sums the orders first. Angular
	multiplicity, derivatives, and Sobolev weights then contribute only polynomial
	factors. The two arguments are thus rowwise and columnwise representations of
	the same off-diagonal propagation geometry.
	

	As a byproduct, our approach applies directly to multiple scattering by circular obstacles in two dimensions, and
	greatly simplifies the corresponding analysis.
	Specifically, up to normalization constants independent of the Fourier order, fixing
	a target mode $m$ gives the row kernel
	\[
	F_{p,m}^{\rm 2D}(\bm{y})
	=
	C_p
	\frac{H_m^{(1)}(k|\bm{y}-\bm{c}_p|)}
	{H_m^{(1)}(ka_p)}
	e^{-im\arg(\bm{y}-\bm{c}_p)},
	\qquad \bm{y}\in\Gamma_q .
	\]
	Since $|\bm{y}-\bm{c}_p|\ge d_{pq}-a_q$, the large-order Hankel ratio yields
	\[
	\limsup_{N\to\infty}
	\|Q_{p,N}A_{pq}\|^{1/N}
	\le
	\frac{a_p}{d_{pq}-a_q}.
	\]
	Conversely, fixing a source mode produces, again up to an
	order-independent normalization,
	\[
	V_{pq}b_{q,m}(\bm{x})
	=
	C_q J_m(ka_q)
	H_m^{(1)}(k|\bm{x}-\bm{c}_q|)
	e^{im\arg(\bm{x}-\bm{c}_q)},
	\qquad \bm{x}\in\Gamma_p ,
	\]
	and $|\bm{x}-\bm{c}_q|\ge d_{pq}-a_p$ gives
	\[
	\limsup_{N\to\infty}
	\|A_{pq}Q_{q,N}\|^{1/N}
	\le
	\frac{a_q}{d_{pq}-a_p}.
	\]
	Thus the same two directed geometric factors follow without expanding
	the interaction into coupled Fourier translation coefficients or
	estimating the resulting hypergeometric sums. Therefore, maximizing over
	ordered pairs recovers
	\[
	\rho_{\rm geo}
	=
	\max_{p\neq q}\frac{a_p}{d_{pq}-a_q}.
	\]
	The sharper first-transfer factor $a_p/d_{pq}$ for plane wave incidence follows in the same
	way from the larger analyticity radius of the isolated disk field.
	
	\section{Numerical experiments}\label{sec:numerics}
	
	In this section, we numerically test the predicted root-asymptotic convergence rates using
	three sound-soft spheres with radii
	\[
	(a_1,a_2,a_3)=(1,0.75,0.5)
	\]
	and centers
	\begin{equation}\label{eq:numerical-centers}
		\bm{c}_1=(0,0,0),\qquad \bm{c}_2=(L,0,0),\qquad
		\bm{c}_3=(0.25L,0.95L,0.18L).
	\end{equation}
	We use $L=2.3,4,7$ for the close, moderate, and far configurations,
	respectively, as illustrated in Figure~\ref{fig:sphere-configurations}. The wavenumbers are $k=0.8,2,4$.
	\begin{figure}[htbp]
		\centering
		\includegraphics[width=\textwidth]{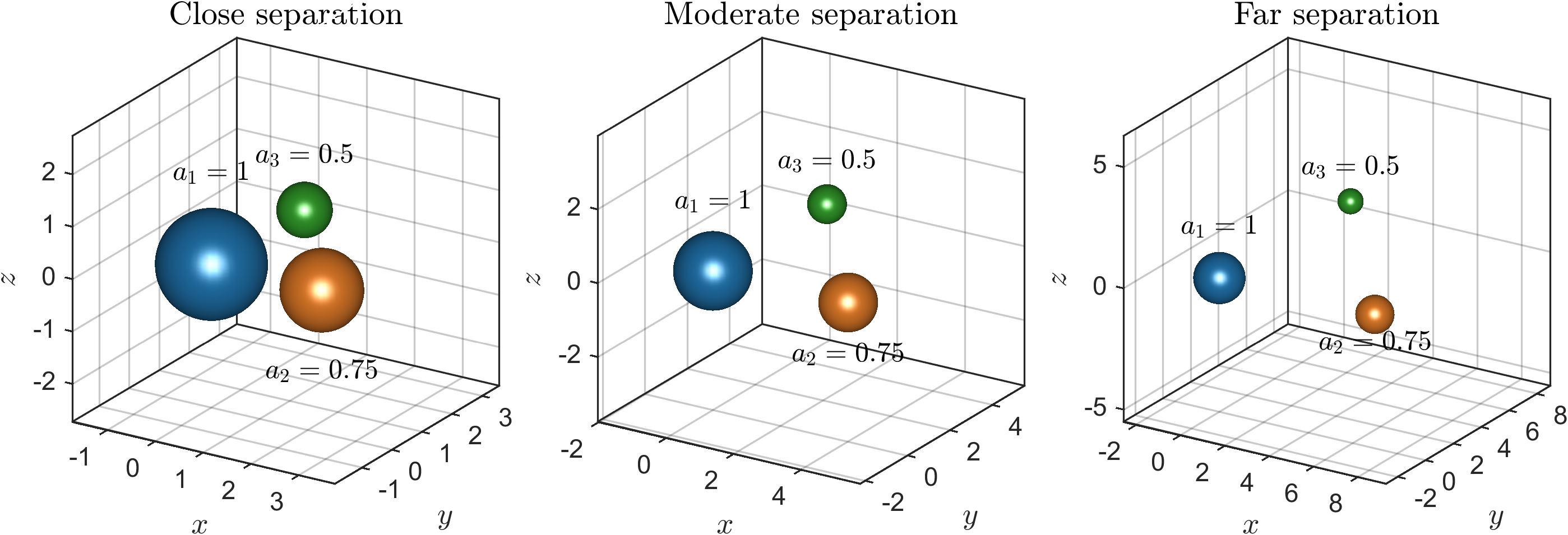}
		\caption{The three different sound-soft spheres in the close, moderate, and
			far configurations.}
		\label{fig:sphere-configurations}
	\end{figure}
	The plane wave direction and point source location are
	\[
	\widehat{\bm d}=\frac{(1,-2,1.5)}{\sqrt{7.25}},
	\qquad \bm{x}_0=(-10,-8,6).
	\]
	For the close configuration, we use the reference degree $N_{\rm ref}=24$
	and plot the range $2\leq N\leq19$. For the moderate and far configurations,
	we use $N_{\rm ref}=20$ and plot $2\leq N\leq15$. Each degree-$N$ truncated
	system is solved independently. For numerical robustness, the translation
	blocks are assembled by direct spherical-harmonic projection rather than by
	explicit Gaunt translation coefficients. The spherical quadrature has polar
	order $N_{\rm ref}+10$ and $2(N_{\rm ref}+10)+1$ azimuthal points. To assess the quadrature error introduced by the projection, we recompute the translation blocks
for the most demanding case, the close configuration with $k=4$ and
$N=19$, using $34\times69$ and $40\times81$ spherical quadrature
nodes.  Let $A_{pq}^{(1)}$ and $A_{pq}^{(2)}$ denote the resulting
blocks.  Then
\[
    \max_{p\ne q}
    \frac{\|A_{pq}^{(2)}-A_{pq}^{(1)}\|_F}
         {\|A_{pq}^{(2)}\|_F}
    =1.26\times10^{-12},
\]
showing that the translation blocks are insensitive to further quadrature refinement at the reported accuracy.

	The computed density error is
	\begin{equation}\label{eq:numerical-error}
		\epsilon_N:=
		\norm{\Phi_{N_{\rm ref}}-\Phi_N}_{\seqsp^{-1/2}},
	\end{equation}
	where the shorter coefficient vector is extended by
	zero-padding. For comparison, we set
	\begin{align}
		\rho_{\rm full}&:=\max\{\rho_{\rm data},\rho_{\rm geo}\},
		\label{eq:rho-full-numerical}\\
		\rho_{\rm first}&:=
		\begin{cases}
			\displaystyle\max_{p\ne q}\frac{a_p}{d_{pq}},
			&\text{plane-wave incidence},\\[2mm]
			\displaystyle\max\left\{\max_p\frac{a_p}{d_{p0}},
			\max_{p\ne q}\frac{a_pd_{q0}}{d_{pq}d_{q0}-a_q^2}\right\},
			&\text{point-source incidence}.
		\end{cases}
		\label{eq:rho-first-numerical}
	\end{align}
	Thus, $\rho_{\rm full}$ is the rigorous root bound for the total error,
	whereas $\rho_{\rm first}$ is the sharper factor for the first-transfer
	component. In Figures~\ref{fig:numerical-plane}
	and~\ref{fig:numerical-point}, the black markers always represent the total
	error $\epsilon_N$.
	
	Our asymptotic theory gives the root factors explicitly but does not provide useful
	prefactors in bounds of the form $C\rho^N$. 
	We estimate
	the empirical root by applying least squares to the logarithms of the last six
	data points that remain above the numerical noise floor:
	\begin{equation}\label{eq:rhohat-epsilon}
		\log\epsilon_N\approx \log C_{\rm fit}
		+N\log\widehat\rho_{\epsilon}.
	\end{equation}
	Only the fitted decay factor $\widehat\rho_\epsilon$, or equivalently the
	fitted slope $\log\widehat\rho_\epsilon$, is shown here, while the fitted intercept is
	not displayed.
	
    \begin{figure}[!htbp]
    \centering

    \begin{subfigure}[t]{0.33\textwidth}
        \centering
        \includegraphics[
            width=\linewidth
        ]{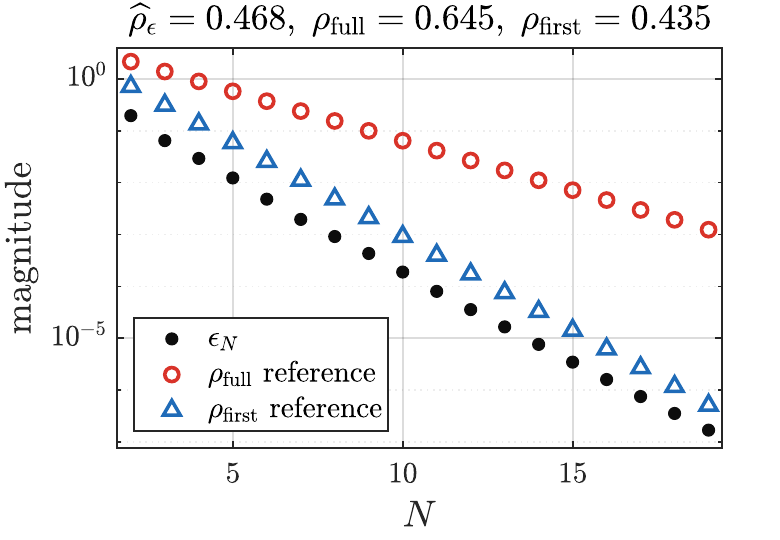}
        \caption{close, $k=0.8$}
    \end{subfigure}%
    \hspace{0.001\textwidth}%
    \begin{subfigure}[t]{0.33\textwidth}
        \centering
        \includegraphics[
            width=\linewidth
        ]{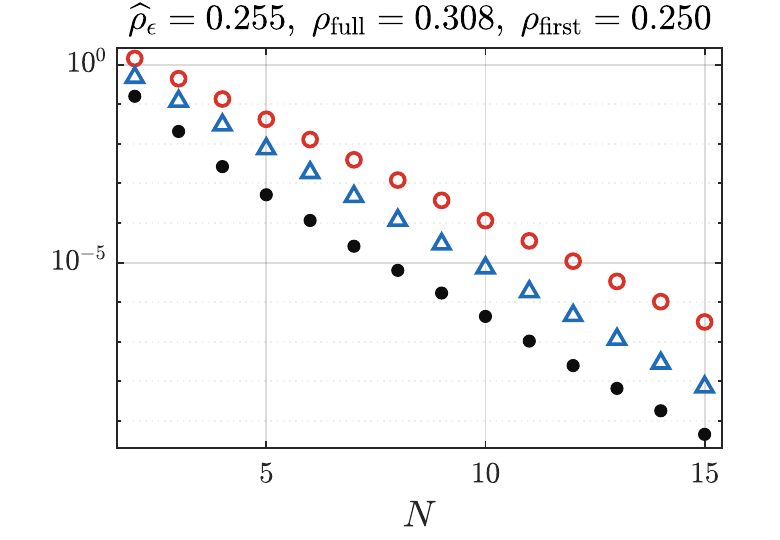}
        \caption{moderate, $k=0.8$}
    \end{subfigure}%
    \hspace{0.001\textwidth}%
    \begin{subfigure}[t]{0.33\textwidth}
        \centering
        \includegraphics[
            width=\linewidth
        ]{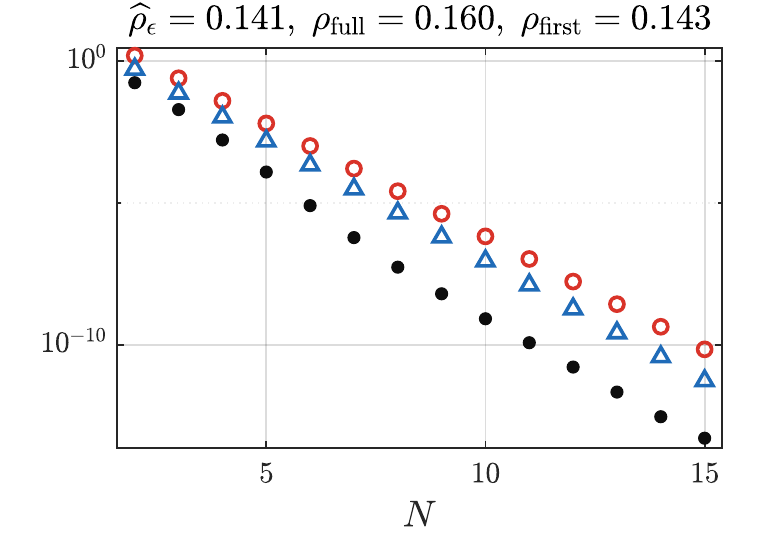}
        \caption{far, $k=0.8$}
    \end{subfigure}

    \vspace{0.5em}

    \begin{subfigure}[t]{0.33\textwidth}
        \centering
        \includegraphics[
            width=\linewidth
        ]{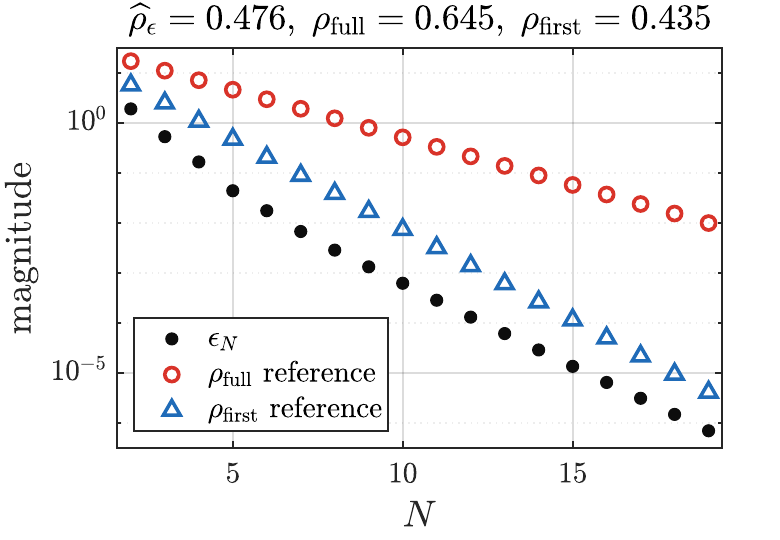}
        \caption{close, $k=2$}
    \end{subfigure}%
    \hspace{0.001\textwidth}%
    \begin{subfigure}[t]{0.33\textwidth}
        \centering
        \includegraphics[
            width=\linewidth
        ]{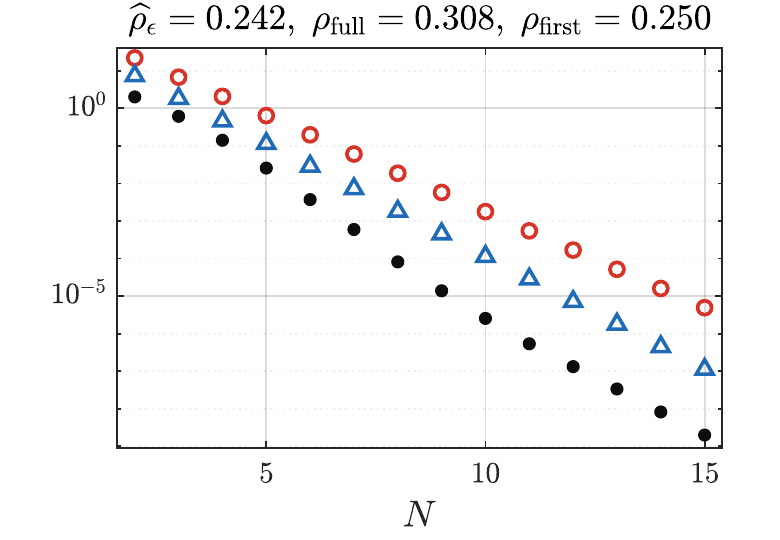}
        \caption{moderate, $k=2$}
    \end{subfigure}%
    \hspace{0.001\textwidth}%
    \begin{subfigure}[t]{0.33\textwidth}
        \centering
        \includegraphics[
            width=\linewidth
        ]{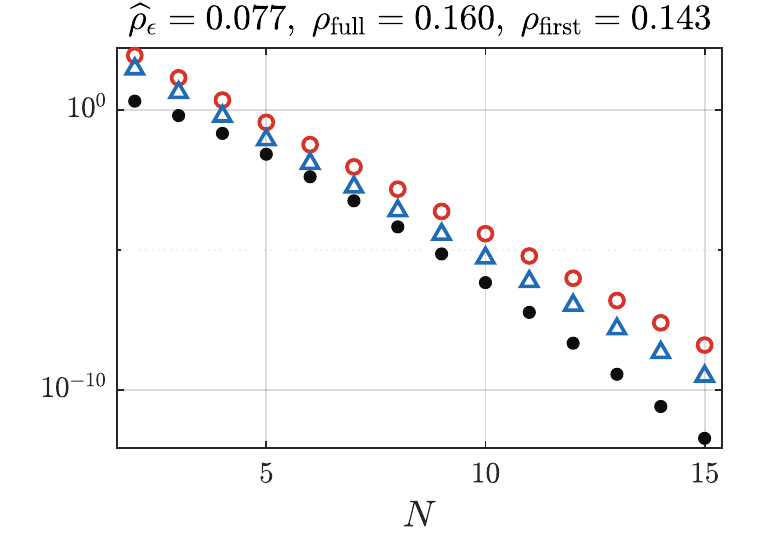}
        \caption{far, $k=2$}
    \end{subfigure}

    \vspace{0.5em}

    \begin{subfigure}[t]{0.33\textwidth}
        \centering
        \includegraphics[
            width=\linewidth
        ]{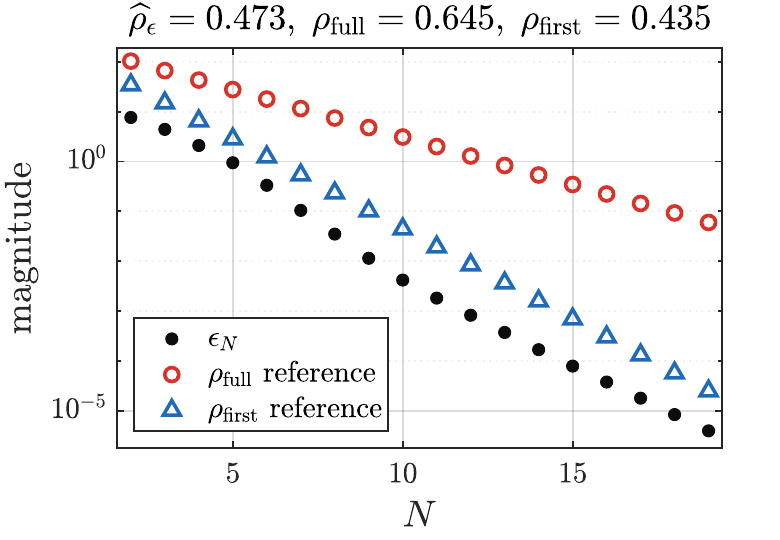}
        \caption{close, $k=4$}
    \end{subfigure}%
    \hspace{0.001\textwidth}%
    \begin{subfigure}[t]{0.33\textwidth}
        \centering
        \includegraphics[
            width=\linewidth
        ]{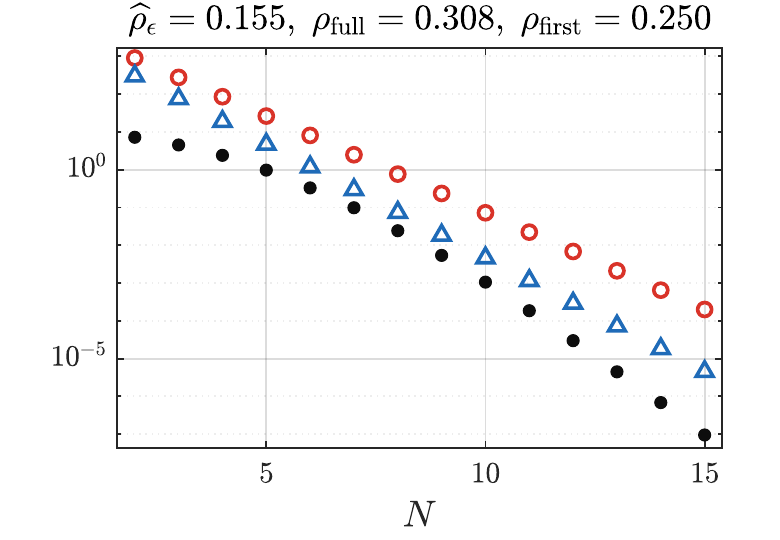}
        \caption{moderate, $k=4$}
    \end{subfigure}%
    \hspace{0.001\textwidth}%
    \begin{subfigure}[t]{0.33\textwidth}
        \centering
        \includegraphics[
            width=\linewidth
        ]{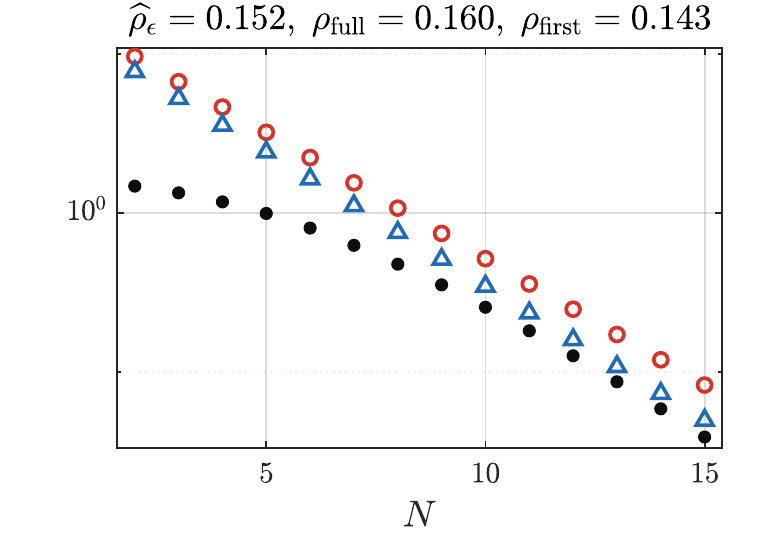}
        \caption{far, $k=4$}
    \end{subfigure}
\caption{Plane wave incidence. Filled black markers show the computed
			density error $\epsilon_N$ against the reference solution. Open red
			circles and open blue triangles show vertically shifted reference
			sequences for $\rho_{\rm full}$ and $\rho_{\rm first}$, respectively. Only their slopes are intended for comparison.
			Rows correspond to $k=0.8,2,4$, and columns correspond to the close,
			moderate, and far configurations.}
		\label{fig:numerical-plane}
\end{figure}

	\subsection{Plane wave incidence} For the configurations described above,
	all plane-wave errors in Figure~\ref{fig:numerical-plane} exhibit spectral
	decay.
	In the close configuration, the fitted roots
	$\widehat\rho_{\epsilon}=0.468$--$0.476$ lie below the rigorous full bound
	$\rho_{\rm full}=0.645$ and are closer to the first-transfer factor
	$\rho_{\rm first}=0.435$. For the moderate configuration at $k=0.8$,
	the fitted value $0.255$ is close to the first-transfer value $0.250$.
	These observations are consistent with both the full convergence theorem and
	the sharper analyticity result identified in Section~\ref{sec:first}, but
	they do not imply that the first-transfer factor bounds the total error. At
	higher wavenumbers, the finite-$N$ curves bend before entering the asymptotic
	geometric regime, reflecting pre-asymptotic Bessel and Hankel factors.
	
    \begin{figure}[!htbp]
    \centering

    \begin{subfigure}[t]{0.33\textwidth}
        \centering
        \includegraphics[
            width=\linewidth
        ]{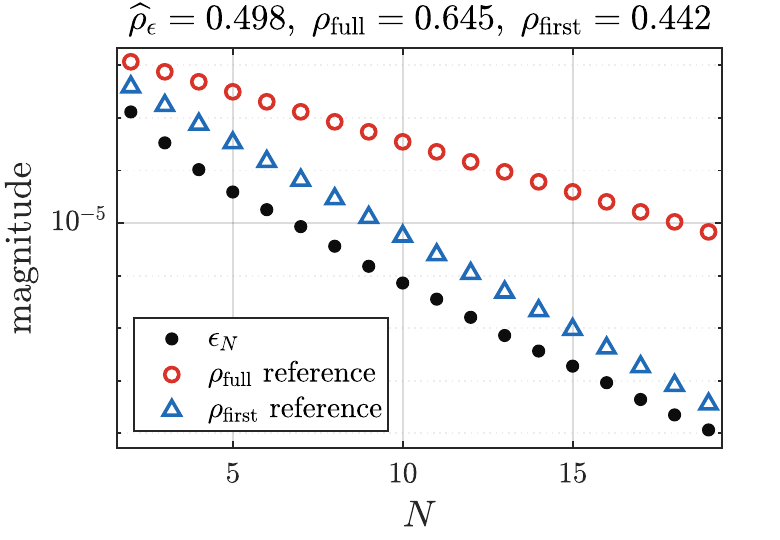}
        \caption{close, $k=0.8$}
    \end{subfigure}%
    \hspace{0.001\textwidth}%
    \begin{subfigure}[t]{0.33\textwidth}
        \centering
        \includegraphics[
            width=\linewidth
        ]{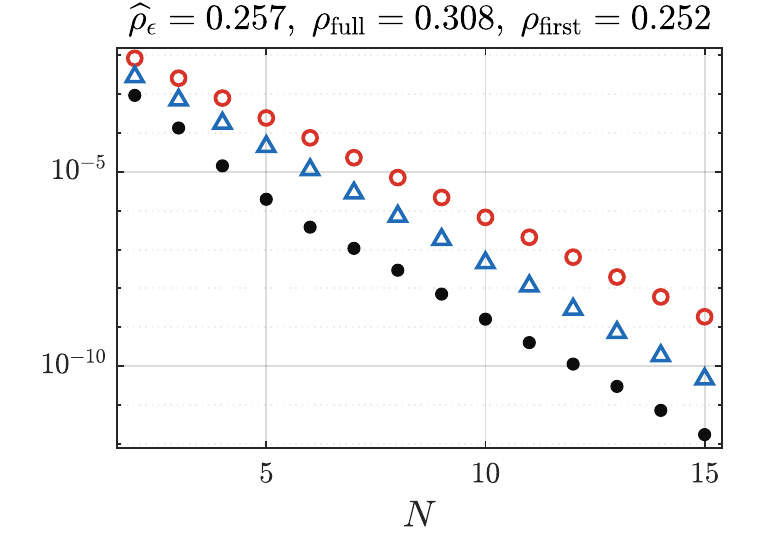}
        \caption{moderate, $k=0.8$}
    \end{subfigure}%
    \hspace{0.001\textwidth}%
    \begin{subfigure}[t]{0.33\textwidth}
        \centering
        \includegraphics[
            width=\linewidth]{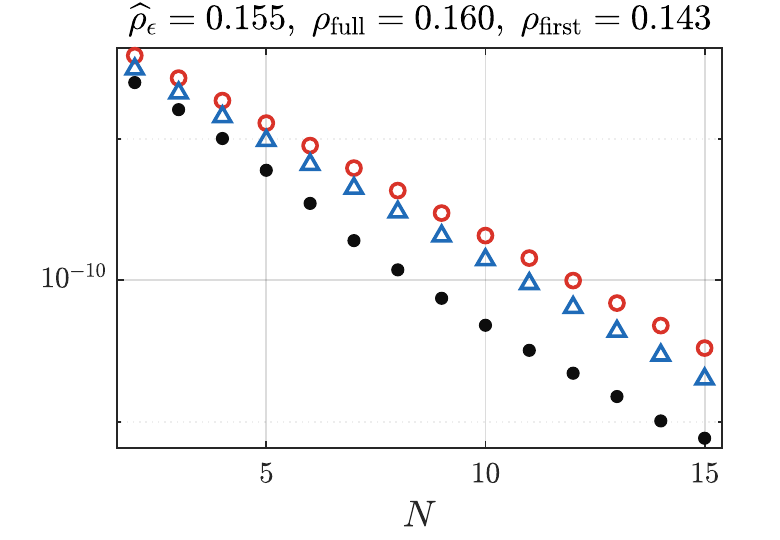}
        \caption{far, $k=0.8$}
    \end{subfigure}

    \vspace{0.5em}

    \begin{subfigure}[t]{0.33\textwidth}
        \centering
        \includegraphics[
            width=\linewidth
        ]{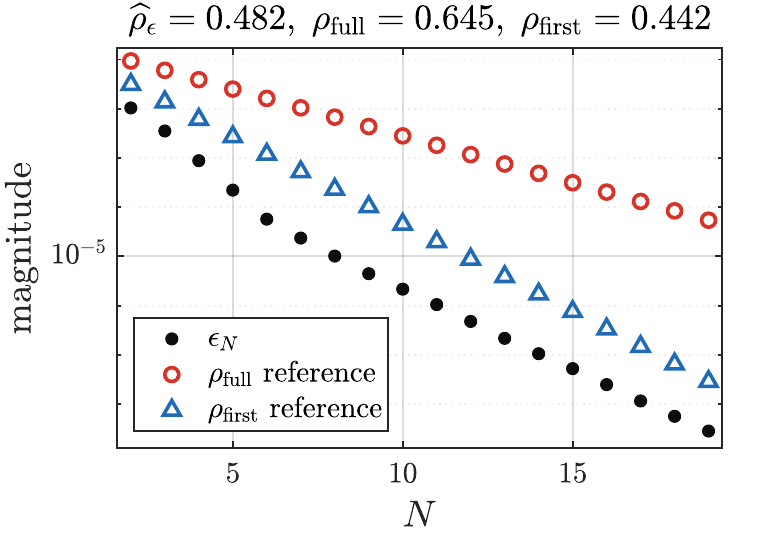}
        \caption{close, $k=2$}
    \end{subfigure}%
    \hspace{0.001\textwidth}%
    \begin{subfigure}[t]{0.33\textwidth}
        \centering
        \includegraphics[
            width=\linewidth
        ]{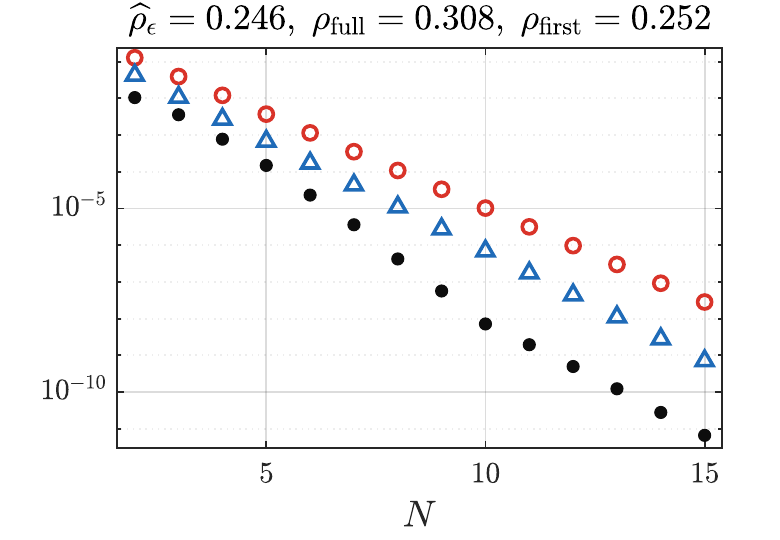}
        \caption{moderate, $k=2$}
    \end{subfigure}%
    \hspace{0.001\textwidth}%
    \begin{subfigure}[t]{0.33\textwidth}
        \centering
        \includegraphics[
            width=\linewidth
        ]{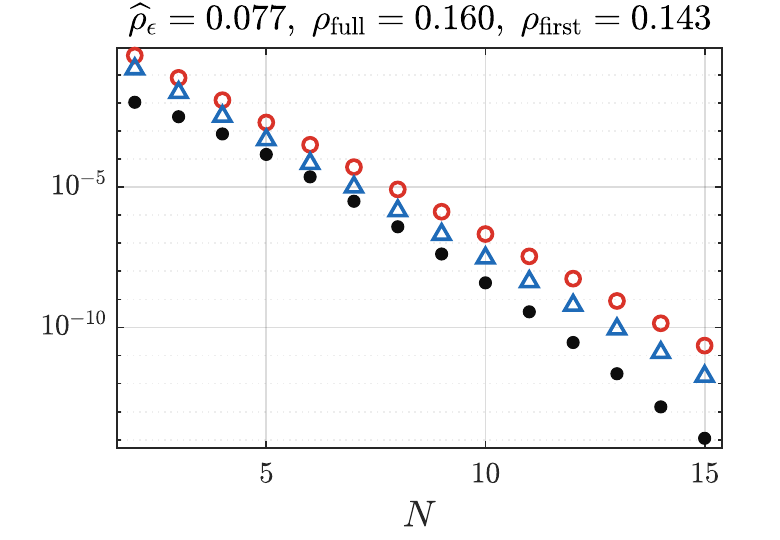}
        \caption{far, $k=2$}
    \end{subfigure}

    \vspace{0.5em}

    \begin{subfigure}[t]{0.33\textwidth}
        \centering
        \includegraphics[
            width=\linewidth
        ]{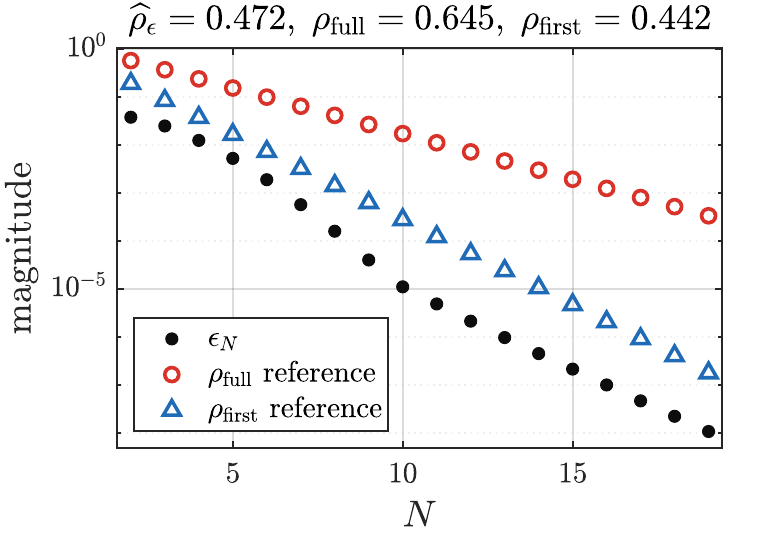}
        \caption{close, $k=4$}
    \end{subfigure}%
    \hspace{0.001\textwidth}%
    \begin{subfigure}[t]{0.33\textwidth}
        \centering
        \includegraphics[
            width=\linewidth
        ]{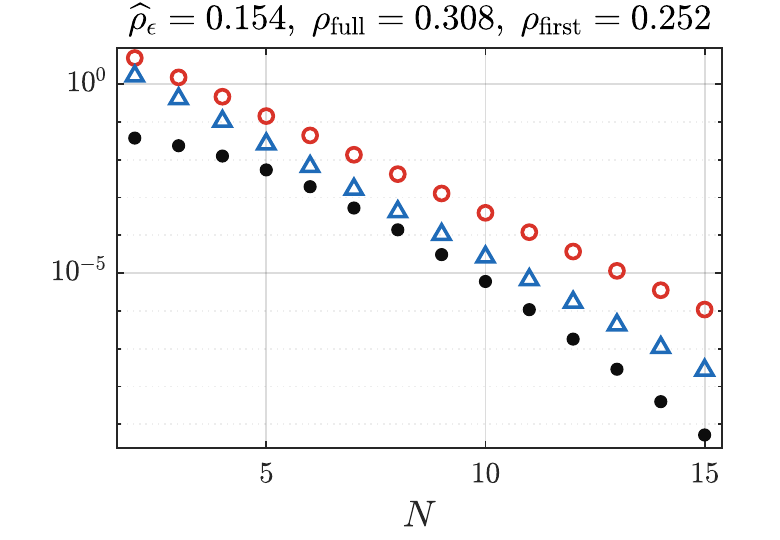}
        \caption{moderate, $k=4$}
    \end{subfigure}%
    \hspace{0.001\textwidth}%
    \begin{subfigure}[t]{0.33\textwidth}
        \centering
        \includegraphics[
            width=\linewidth
        ]{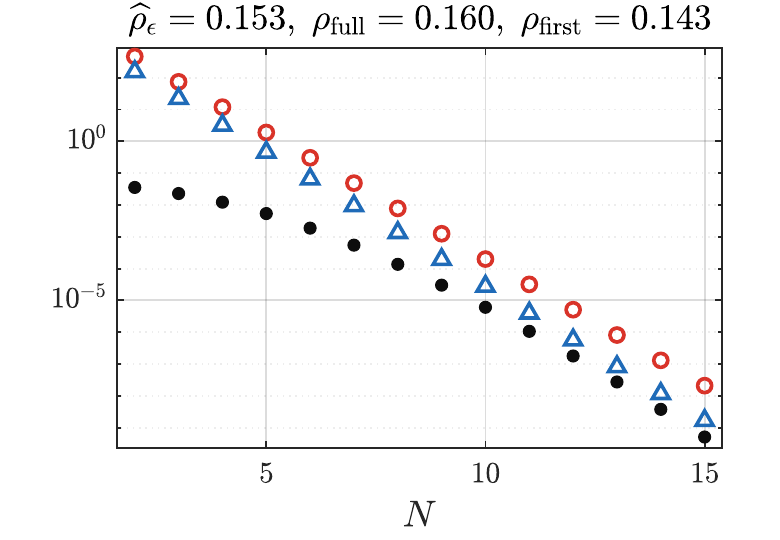}
        \caption{far, $k=4$}
    \end{subfigure}

    \caption{Point source incidence. Filled black markers show the computed
			density error $\epsilon_N$ against the reference solution. Open red
			circles and open blue triangles show the vertically shifted full-rate
			and inversion-radius first-transfer reference sequences, respectively. Only their slopes are intended for comparison.}
		\label{fig:numerical-point}
\end{figure}

	\subsection{Point source incidence}
	For the chosen point source, $\rho_{\rm data}=0.0707$ is smaller than the
	geometric factor in all three configurations, so the full bound is determined
	by $\rho_{\rm geo}$. Figure~\ref{fig:numerical-point} therefore exhibits the
	same close-to-far hierarchy as the plane wave experiment. In the close
	configuration, the first-transfer value is $0.442$, while the three fitted
	total-error roots are $0.498$, $0.482$, and $0.472$, all substantially
	closer to the first-transfer value than to the full value $0.645$. For the
	moderate configuration at
	$k=0.8$, the theoretical first-transfer value $0.252$ and the fitted value
	$0.257$ nearly coincide. In the far configuration the error reaches the
	double-precision floor rapidly, so finite-interval fits may appear faster
	than either theoretical reference. Overall, the experiments agree with the
	predicted spectral convergence and illustrate the conservative nature of
	the worst-case full bound.
	
	\section{Conclusion}\label{sec:conclusion}
	
	We have established spectral convergence of the fully coupled
	three-dimensional multisphere MEM in the natural density space
	$\seqsp^{-1/2}$. Our main full error estimate is
	\[
		\limsup_{N\to\infty}
		\norm{\Phi-\Phi_N}_{\seqsp^{-1/2}}^{1/N}
		\leq
		\max\left\{\rho_{\rm data},
		\max_{p\ne q}\frac{a_p}{d_{pq}-a_q}\right\}.
	\]
	Theorem~\ref{thm:first-transfer} further provides sharper convergence factors
	for the first-transfer component in terms of the isolated sphere analyticity
	radii.

	The proof rests on a decomposition of the interaction truncation into
	target-side and source-side high-degree tails. The former is controlled by a
	projected Green-kernel representation, whereas the latter is estimated through
	the physical value and gradient energies of a translated degree-$n$
	spherical-wave family. Summing over all orders before taking the large-degree
	estimate avoids coefficient-wise bounds on the Gaunt couplings. The remaining
	dimensional and Sobolev effects contribute only polynomial factors and therefore
	do not alter the geometric root rate. The same framework yields a shorter
	derivation of the corresponding two-dimensional bounds and identifies
	propagation distance and analyticity radius as the geometric origins of the
	convergence factors. Numerical experiments for plane wave and point source
	incidence are consistent with the predicted spectral decay and the theoretical
	root bounds. Future work includes the analysis of MEM for multiple scattering by nonspherical
	objects and more complex scattering problems in layered media.

	
	\appendix
	
	\section{Addition theorems for spherical harmonics}\label{app:spherical-identities}
	
	The spherical harmonic addition theorem states that \cite{DLMF}
	\begin{equation}\label{eq:addition-theorem}
		\sum_{m=-n}^nY_n^m(\bm{\omega})\overline{Y_n^m(\bm{\omega}')}
		=\frac{2n+1}{4\pi}P_n(\bm{\omega}\cdot\bm{\omega}').
	\end{equation}
	On the diagonal, it holds
	\begin{equation}\label{eq:value-shell}
		\sum_{m=-n}^n\abs{Y_n^m(\bm{\omega})}^2=\frac{2n+1}{4\pi}.
	\end{equation}
	Since $-\Delta_{\Sph}Y_n^m=n(n+1)Y_n^m$, differentiating \eqref{eq:addition-theorem} and taking the diagonal gives
	\begin{equation}\label{eq:gradient-shell}
		\sum_{m=-n}^n\abs{\nabla_{\Sph}Y_n^m(\bm{\omega})}^2
		=\frac{(2n+1)n(n+1)}{4\pi}.
	\end{equation}
    Equations \eqref{eq:value-shell}--\eqref{eq:gradient-shell} are the formulas that replace the order-index sums in the translation coefficients.
	Let $\bm{c}\in\mathbb R^3$ be the expansion center, and set
	\[
		r_{\bm{x}}:=|\bm{x}-\bm{c}|,\qquad
		r_{\bm{y}}:=|\bm{y}-\bm{c}|,\qquad
		r_<:=\min\{r_{\bm{x}},r_{\bm{y}}\},\qquad
		r_>:=\max\{r_{\bm{x}},r_{\bm{y}}\},
	\]
	with $\widehat{\bm r}_<,\widehat{\bm r}_>$ denoting the corresponding unit
	directions. Then, for $r_{\bm{x}}\neq r_{\bm{y}}$, the spherical wave addition
	theorem gives \cite{Stein1961,EptonDembart1995}
	\begin{equation}\label{eq:green-expansion}
		\Phi_k(\bm{x},\bm{y})=\ii k\sum_{n=0}^\infty\sum_{m=-n}^n
		j_n(kr_<)h_n^{(1)}(kr_>)
		Y_n^m(\widehat{\bm r}_<)\overline{Y_n^m(\widehat{\bm r}_>)}.
	\end{equation}

	\bibliographystyle{abbrv}
	\bibliography{bib.bib}
	
\end{document}